\documentclass[11pt]{amsart}
\usepackage{graphicx}
\usepackage{amsfonts}
\usepackage{indentfirst,bm,amsmath}
\usepackage{amsthm,amssymb}
\usepackage{mathrsfs}
\usepackage{hyperref}

\usepackage{subfig}

\newtheorem{theorem}{Theorem}[section]
\newtheorem{corollary}[theorem]{Corollary}
\newtheorem{lemma}[theorem]{Lemma}
\newtheorem{proposition}[theorem]{Proposition}
\newtheorem{definition}[theorem]{Definition}
\newtheorem{remark}[theorem]{Remark}
\newtheorem{example}[theorem]{Example}
\newtheorem{question}[theorem]{Question}

\begin{document}
	\title[]{Tropical Cartan's second main theorem for hyperplanes in general position}
	
	\author[Y. T. Wang \and T. B. Cao]{ Yuting Wang \and Tingbin Cao}
	
	\address[Yuting Wang and Tingbin Cao]{Department of Mathematics, Nanchang University, Nanchang city, Jiangxi 330031, P. R. China}
	\email{ytwang@email.ncu.edu.cn, tbcao@ncu.edu.cn (the corresponding author)}
	
	\thanks{The first author is supported by the National Natural Science Foundation of China (\#12571082), the Jiangxi Natural Science Foundation (\#20232ACB201005) and the Shandong Natural Science Foundation (\#ZR2024MA024)}
	\date{}
	
	
	
	\begin{abstract}
We prove a tropical analogue of Cartan's second main theorem for holomorphic curves intersecting hyperplanes in general position——a setting that was not fully resolved by previous tropical Nevanlinna theory. Two versions are obtained. The first (Theorem \ref{T4.1.10}) requires subnormal growth and involves the tropical Casorati determinant. The second and main version (Theorem \ref{T4.1.2}) is completely free of growth conditions and exceptional sets; it replaces the Casorati term by the sum of the counting functions of the curve's components, yielding an inequality valid for every $r.$ The proof uses a tropical Cramer theorem, bypassing the logarithmic derivative lemma. This improves upon previous results by Korhonen-Tohge and Cao-Zheng, where the coefficient could be suboptimal even under the general position hypothesis.
We also clarify the relation between different notions of linear independence, and present the first counterexample to the truncated second main theorem in the tropical setting (Example \ref{E5.7}).
\end{abstract}
	
	\maketitle
	\tableofcontents

\section{Introduction and main results}
The main goal in this paper is to establish the theory for tropical holomorphic maps intersecting tropical hyperplanes in general position over the tropical semi-field. Recall that in 1933, H. Cartan \cite{14} established the classical Nevanlinna theory for holomorphic curves over the complex field and proved the second main theorem with hyperplanes in general position, namely,
\begin{eqnarray*}
(q - n - 1)T_f(r)\leq\sum_{j = 1}^q N_f(r, H_j)-N(r, \frac{1}{W(f)})+O(\log^+(rT_f(r)))\end{eqnarray*} holds for all $r$ outside a set with finite logarithmic measure, where the coefficient $q-n-1$ is optimal. Here $N(r, \frac{1}{W(f)})$ is the counting function of the Wronskian determinant of $f.$ For a better error term in Cartan theory, see \cite{Ye95}.\par

The development of Nevanlinna theory over tropical geometry has seen significant progress. In 2009, Halburd and Southall \cite{1} first proposed Nevanlinna's characteristic function, the Poisson-Jenson formula and the first main theorem for tropical meromorphic functions of one real variable, with finite order. In 2011, for modified tropical meromorphic functions with arbitrary real powers in tropical Laurent series, Laine and Tohge \cite{5} proved the Nevanlinna's second main theorem for tropical meromorphic functions with hyper-order strictly less than one.\par

In 2016, Korhonen and Tohge \cite{6} considered the degeneracy of the set of tropical linear combinations of $\{f_0,\dots,f_n\}$ in the sense of Gondran-Minoux \cite{gm1, gm2}, and attempted to obtain a tropical version of Cartan's second main theorem for tropical holomorphic curves into tropical projective space $\mathbb{TP}^{n}$ with hyperorder strictly less than one. In 2025, Cao and Zheng \cite{7} extended Korhonen-Tohge's second main theorem to the case of tropical hypersurfaces, in which the growth condition of tropical holomorphic curves is extended to minimal hypertype (or called subnormal growth). Very recently, Cao and Peng \cite{Cao-peng} successfully extended these results to several tropical variables. To state the main results,  we first recall that a tropical hyperplane consists of all the roots of a tropical linear polynomial in \( \mathbb{TP}^n \) (see \cite[Proposition 3.3]{mik}). \par

\begin{definition} A tropical hyperplane is a subset of the $n$-dimensional tropical projective space $\mathbb{TP}^n$ of the form
\begin{equation}\label{(3)}
H = \left\{x \in \mathbb{TP}^n \left| \, P(x)=\max_{i=0}^{n}(a_i + x_i) \text{ is attained at least twice}\right.\right\}.
\end{equation}
Denote by $a=(a_0, a_1, \dots, a_n)\in\mathbb{TP}^n$ the coefficient vector in terms of the tropical hyperplane $H.$ \end{definition}\par

Let $f_0, \dots, f_n$ be given tropical entire functions. Denote by $Q$  a collection of tropical linear combinations $g=\bigoplus_{k = 0}^n a_{k}\otimes f_{k},$ where the coefficients $a_0, \ldots, a_n\in\mathbb{T}$ and not all identical to $0_{\mathbb{T}}.$ The $\ell(g)$ represents the shortest length of the representation of $g\in Q$ defined by
\[
\ell(g)=\min\left\{j\in\{1,\dots,n + 1\}:g = \bigoplus_{i = 1}^j a_{k_i}\otimes f_{k_i}\right\}
\]
where $a_{k_i}\in\mathbb{R}$ with integers $0\leq k_1<k_2<\dots<k_j\leq n$. If $\ell(g)=n + 1$ for a tropical linear combination $g$ of $f_0,\dots,f_n$, then $g$ is said to be \emph{complete}. Korhonen and Tohge \cite{6} defined the degree of degeneracy of $Q$ by
\[
ddg(Q):=card(\{g\in Q:\ell(g)<n + 1\}).
\] The $ddg(Q)$ means the number of non-complete tropical linear combinations in the set $Q.$ The basic definitions and notation for a tropical holomorphic map into tropical projective spaces such as the characteristic function $T_f(r),$ counting function $N(r, f)$ will be introduced in the next section.\par

\begin{theorem}\label{T1.2} \cite{7,6} Let $q$ and $n$ be positive integers with $q\geq n$. Assume that the tropical holomorphic map $f:\mathbb{R}\to\mathbb{TP}^n$ is tropically linearly nondegenerate (i.e., its image is not contained in any tropical hyperplane). Assume that $H_{j}$ are tropical hyperplanes in $\mathbb{TP}^n$ defined by tropical linear polynomials  $P_j$ ($j = 1,\dots,q$), respectively. If $\lambda = d d g(\{P_{n + 2}\circ f,\dots,P_{q}\circ f\})$ and
$\limsup_{r\to\infty}\frac{\log T_f(r)}{r}=0,$
then
\begin{align*}
(q - n -& 1-\lambda)T_f(r)\\
\leq&\sum_{j = 1}^{q}N\left(r,\frac{1_{\mathbb{T}}}{P_j\circ f}\oslash\right)-N\left(r,\frac{1_{\mathbb{T}}}{C(P_1\circ f,\dots,P_{n + 1}\circ f)}\oslash\right)+o(T_f(r))\\
\leq& \sum_{j = n + 2}^{q}N\left(r,\frac{1_{\mathbb{T}}}{P_j\circ f}\oslash\right)+o(T_f(r)),
\end{align*}
where $r$ approaches infinity outside an exceptional set of zero upper density measure.
\end{theorem}

However, Theorem \ref{T1.2} involve a degeneracy parameter $\lambda$ that can reduce the coefficient even when the tropical hyperplanes are in general position (Definition \ref{D1.3} below). The following example illustrates this phenomenon. \par

\begin{example}Set the tropical hyperplanes \(H_1, \dots, H_4\) in \(\mathbb{TP}^2\) defined by tropical linear polynomials $P_j$ ($j = 1,\dots, 4$),  respectively, with their coefficient vectors are
\[
a_1=(1,0_{\mathbb{T}},0_{\mathbb{T}}),\quad a_2=(0_{\mathbb{T}},1,1),\quad a_3=(0_{\mathbb{T}},0_{\mathbb{T}},1),\quad a_4=(1,1,0_{\mathbb{T}}),
\] respectively. For any tropically linearly nondegenerate holomorphic curve  \(f = [f_0:f_1:f_2]:\mathbb{R}\to\mathbb{TP}^2\) satisfying the subnormal growth condition, we have $\lambda=ddg(\{P_{4}\circ f\})=1\neq 0.$  Then it follows from Theorem \ref{T1.2} that we have a trivial estimate
\begin{equation*}
\begin{split}
0&=(4-2-1-1)T_{f}(r)\\&\leq\sum_{j = 1}^{4}N\left(r,\frac{1_{\mathbb{T}}}{P_j\circ f}\oslash\right)-N\left(r,\frac{1_{\mathbb{T}}}{C(P_1\circ f,\dots,P_{3}\circ f)}\oslash\right)+o(T_f(r))\\
&\leq N(r, \frac{1_{\mathbb{T}}}{P_4\circ f} \oslash)+o(T_f(r)).
\end{split} 
\end{equation*} Thus the optimal constant $q-n-1=1$ is not attained despite the tropical hyperplanes being in general position. \end{example}
\par

\begin{definition}\label{D1.3}
Given tropical hyperplanes $H_1, \dots, H_q$ in $\mathbb{TP}^{n}.$ We say that $H_1,$ $\dots,$ $H_q$ are in general position if for any injective map $\mu : \{0, 1, \dots, n\} \to \{1 \dots, q\}$, one has $$H_{\mu(0)}\cap \dots\cap H_{\mu(n)}=\varnothing.$$
\end{definition}

This raises two natural questions.

\begin{question}\label{Q1}
Can one obtain a second main theorem for tropical hyperplanes in general position with the optimal coefficient $q-n-1$, independent of any degeneracy parameter?
\end{question}

\begin{question}\label{Q2}
Is it possible to eliminate the subnormal growth condition  $\limsup_{r\to\infty}\frac{\log T_f(r)}{r}=0$ and the accompanying exceptional set entirely?
\end{question}

In this paper we answer both questions affirmatively. For convenience, we denote by $C(f):=C(f_0,\dots,f_n)$ the tropical Casorati determinant (see Section 2).\par

Our first result (Theorem \ref{T4.1.10}) improves Theorem \ref{T1.2} by replacing the Casorati of the pulled‑back hyperplanes with the simpler Casorati of $f$ itself, and it achieves the optimal coefficient $q-n-1$ under the general position hypothesis. Although Theorem \ref{T4.1.10}) closes the structure of the classical Cartan theorem: it involves the tropical Casorati determinant $C(f)$ as the ramification term and and holds only outside an exceptional set of zero upper density, it still requires the subnormal growth condition $\limsup_{r\to\infty}\frac{\log T_f(r)}{r}=0.$ Note that (see Lemma \ref{C2.13}) the tropical hyperplanes $H_1, \dots, H_q$  are in general position, if and only if, for any injective map $\mu : \{0, 1, \dots, n\} \to \{1 \dots, q\}$,  the corresponding coefficient vectors $a_{\mu(0)}, \dots, a_{\mu(n)}$ are tropically linearly independent (which is different from the  Gondran-Minoux sense, see the details in the third section). Thanks to Lemma \ref{C2.13}, the general position condition of tropical hyperplanes is equivalent to the tropical linear independence of their coefficient vectors. This equivalence allows us to apply the tropical Cramer theorem (Theorem \ref{The2.20}) and the tropical determinant inequality (Theorem \ref{T3.12}) to derive the desired second main theorems. Actually,  we obtain a general form of the second main theorem with arbitrary tropical hyperplanes (Theorem \ref{T4.1.7}) to get the following result. \par

\begin{theorem}\label{T4.1.10} Assume that the tropical holomorphic map \(f = [f_0:\dots:f_n]:\mathbb{R}\to\mathbb{TP}^n\) is tropically linearly nondegenerate. Let \(H_1,\dots,H_q\) be tropical hyperplanes in \(\mathbb{TP}^n\) in general position defined by tropical linear polynomials $P_j$ ($j = 1,\dots,q$), respectively. If $\limsup_{r \to \infty} \frac{\log T_f(r)}{r}=0,$
then
\begin{equation}\label{eg1.6}
\begin{split}
(q - n - 1)T_{f}(r)\leq\sum_{j = 1}^{q} N(r, \frac{1_{\mathbb{T}}}{P_j\circ f}\oslash) -N(r, \frac{1_{\mathbb{T}}}{C(f)}\oslash)+ o(T_f(r)),
\end{split}
\end{equation}
where $r$ tends to infinity outside of a set of zero upper density measure.
\end{theorem}

We give an example to show that the growth condition in Theorem \ref{T4.1.10} can not be removed.\par

\begin{example}\label{ex1} Let \(f(x)=(0, e_2(x)): \mathbb{R}\rightarrow\mathbb{TP}^1,\) where $$e_2(x)=2^{[x]}(x-[x])+\sum^{[x]-1}_{j=-\infty}2^j=2^{[x]}(x-[x]+1),$$ and let $P_1(x_0, x_1):=x_0, P_2(x_0, x_1):=x_1, P_3(x_0, x_1):=x_0\oplus x_1.$ Clearly, the three tropical hyperplanes from the $P_1, P_2$ and $P_3$ are located in general position in $\mathbb{TP}^1.$ By \cite[proposition 1.24]{Risto15}, we know that $T_f(r)=T(r, e_2(x))+O(1)=\frac{1}{2}(r-[r]+1+o(1))2^{[r]}+O(1).$ Clearly, $\limsup_{r \to \infty} \frac{\log T_f(r)}{r}\neq 0.$ However, since
$$C(f)=e_2(x+1)=2e_2(x),$$ we have $$\sum_{j = 1}^{3}N(r, \frac{1_{\mathbb{T}}}{P_j\circ f}\oslash) -N(r, \frac{1_{\mathbb{T}}}{C(f)}\oslash)\equiv0.$$ Hence \eqref{eg1.6} does not happen. This implies that we cannot drop growth condition in Theorem \ref{T4.1.10}.\end{example}\par

Second, and more importantly, we completely affirm both Question \ref{Q1} and Question \ref{Q2} by establishing a second main theorem that requires neither growth conditions nor exceptional sets (Theorem \ref{T4.1.2}). In this version, the Casorati term is replaced by the sum of the counting functions of the components $f_0,\dots,f_n.$  The proof avoids the logarithmic derivative lemma entirely and instead uses a tropical Cramer theorem (Theorem \ref{The2.20}) to obtain pointwise linear estimates. To our knowledge, this is the first result in tropical Nevanlinna theory that is completely free of analytic growth restrictions.\par

\begin{theorem}\label{T4.1.2}
Assume that the tropical holomorphic map \(f = [f_0 : \dots : f_n]: \mathbb{R}\to\mathbb{TP}^n\) is tropically linearly nondegenerate. If \(H_1,\dots,H_q\) are tropical hyperplanes in \(\mathbb{TP}^n\) in general position defined by tropical linear polynomials $P_j$ ($j = 1,\dots,q$), respectively, then
\[
(q - n - 1)T_{f}(r)\leq\sum_{j = 1}^{q} N(r, \frac{1_{\mathbb{T}}}{P_j\circ f}\oslash) -\sum_{j = 0}^{n}N(r, 1_{\mathbb{T}}\oslash f_j)+ O(1).
\]
\end{theorem}

We give another example to show that the equality in Theorem \ref{T4.1.2} can hold.\par

\begin{example}\label{ex2} Let \(f(x)=(f_0, f_1): \mathbb{R}\to\mathbb{TP}^{1}\) be a non-constant tropical holomorphic map, and let $P_1, P_2, P_3$ be given in Example \ref{ex1}. $$\sum_{j = 1}^{3}N(r, \frac{1_{\mathbb{T}}}{P_j\circ f}\oslash) -\sum_{j=0}^{1}N(r, \frac{1_{\mathbb{T}}}{f_j}\oslash)=N(r, \frac{1_{\mathbb{T}}}{P_3\circ f}\oslash)\leq T_{f}(r)+O(1).$$ Hence the inequality in Theorem \ref{T4.1.2}  becomes an equality $T_f(r)=N(r, \frac{1_{\mathbb{T}}}{P_3\circ f}\oslash)+O(1).$ This implies that conclusion in Theorem \ref{T4.1.2} is sharp.\end{example}\par

As a consequence, we obtain an identity for complete tropical hyperplanes (all coefficients real), which generalizes a one‑dimensional result of Halonen, Korhonen and Filipuk \cite[Corollary 3.7]{8}. Tropical linear combination $P_3\circ f$ in Example \ref{ex2} is complete, and satisfies this equality.\par

\begin{theorem}\label{Tt3.12}
 Assume that the tropical holomorphic map $f=[f_0, f_1, \dots, f_n]: \mathbb{R}\to\mathbb{TP}^n$ is tropically linearly nondegenerate. If $H$ is a tropical hyperplane in \( \mathbb{TP}^n \) defined by the tropical polynomial
\begin{eqnarray*} P(x):=(a_0 \otimes x_0)\oplus \dots \oplus ( a_n \otimes x_n)
\end{eqnarray*} such that all $a_0, \dots, a_n \in\mathbb{R},$ then
\[
T_f(r) = N\left( r, \frac{1_{\mathbb{T}}}{P\circ f} \oslash \right) + O(1).
\]
\end{theorem}

For a tropical holomorphic curve \(f:\mathbb{R}\to\mathbb{TP}^n\) intersecting a tropical hyperplane \(H\) given by a tropical linear polynomial $P$, its defect is defined by
\[
\delta_f(H):=\liminf_{r\to\infty}\frac{m_f(r,H)}{T_f(r)}=1 - \limsup_{r\to\infty}\frac{N(r,\frac{1_{\mathbb{T}}}{P\circ f}\oslash)}{T_f(r)}.
\]
Thus by Theorem \ref{T4.1.2} and Theorem \ref{Tt3.12}, we immediately have the following defect relations.\par

\begin{theorem}\label{T4.13}[Defect relation]
Assume that the tropical holomorphic map \(f = [f_0:\dots:f_n]:\mathbb{R}\to\mathbb{TP}^n\) is tropically linearly nondegenerate. Let \(H_1,\dots,H_q\) be tropical hyperplanes in \(\mathbb{TP}^n\) in general position defined by tropical linear polynomials $P_j$ ($j = 1, \dots, q$), respectively.
Then
$$\sum_{j=1}^q\delta_f(H_j)\leq n+1.$$ Furthermore,  we have  $\delta_f(H)=0$ for any tropical hyperplane $H$ defined by a tropical linear polynomial with all coefficients in $\mathbb{R}.$
\end{theorem}

To apply Theorem \ref{T4.1.2}, we give an improvement (Theorem \ref{E5.4}) of tropical Nevanlinna's second main for meromorphic functions due to Laine-Tohge \cite{5}. Moreover, we show that the truncated second main theorem, which holds in the classical complex case, \textbf{fails} in the tropical setting (Example \ref{E5.7}), answering a question raised by Laine and Tohge \cite{5}.\par

We would like to point out that our main results such as Theorems\ref{T4.1.2}, \ref{T4.1.10} and \ref{Tt3.12}  in this paper can be generalized to the case of several variables in terms of the technique due to Cao and Peng \cite{Cao-peng}.\par

This paper is organized as follows. In the second section, we will briefly introduction some definitions and notation of the tropical Nevanlinna theory. In the third section, we will give some important lemmas on tropically linearly independent and tropical Cramer theorem. We oberseve that tropically linearly dependent is different from the linearly dependent in the Gondran-Minoux sense. Lemma \ref{L3.3} shows that a matrix $A$ is tropically singular if and only if the column vectors of it are tropically linearly dependent. This implies an identical condition on tropical hyperplanes in general position (see Lemma \ref{C2.13}). From the tropcial Gramer theorem, we give a Cramer system's upper bound (Theorem \ref{The2.20}) which plays a key role in the proof of our second main theorem.\par

The proofs of our second main theorems are given in the fourth section. In Subsection 4.1, we  first obtain Theorem \ref{T4.1.7} which is a general second main theorem with arbitrary tropical hyperplanes. After proposing a tropical version of product to the sum estimate (Theorem \ref{L3.7}) and then  we prove Theorem \ref{T4.1.10} by Theorem \ref{T4.1.7}.  In Subsection 4.2,  bypassing the lemma of logarithmic derivative, and directly proceeding from the property that tropical hyperplanes in general position, we prove Theorem \ref{T4.1.2}. In addition, we obtain a relationship
\(N(r, 1_{\mathbb{T}} \oslash C(f))\leq \sum_{j = 0}^{n}N(r, 1_{\mathbb{T}} \oslash f_j)+o(T_f(r))\)  (see Proposition \ref{C3.9}), and give alternative proof of Theorem \ref{T4.1.10} by Theorem \ref{T4.1.2}.\par

In the fifth section, we show that Theorem \ref{T4.1.2} improves and generalizes the Laine-Tohge's second main theorem for tropical meromorphic functions, without any growth condition and exceptional set. Furthermore, we provide an example (Example \ref{E5.7}) to affirm the Laine-Tohge's conjecture that the truncated Nevanlinna's second main theorem in the tropical setting is not true. \par

In the sixth section, we give Theorem \ref{CC4.10}) to relate the general position condition to the degeneracy parameters of Korhonen-Tohge.\par

\section{Preliminaries on tropical Nevanlinna theory} In this section, we will introduce the preliminaries of the tropical semiring, and some basic notation in tropical Nevanlinna theory (see references \cite{1, 12, 5, 9}). Tropical (max-plus) addition and multiplication are defined, respectively, by
\[
x\oplus y:=\max(x,y)
\]
and
\[
x\otimes y:=x + y.
\] In addition, we use the notation \(x\oslash y:=\frac{x}{y}\oslash=x - y\) and \(x^{\otimes \alpha}:=\alpha x\) for \(\alpha\in\mathbb{R}\). Denote by \(\mathbb{R}\cup\{-\infty\}:=\mathbb{T}\). The identity elements for the tropical operations are \(0_{\mathbb{T}}=-\infty\) for addition and $1_{\mathbb{T}}=0$ for multiplication.  \par

\subsection{Tropical meromorphic functions} First, we recall the definition of a tropical meromorphic function, as follows.
\begin{definition}\cite[Definition 1.1]{Risto15}\label{D1.11} A tropical meromorphic function means a continuous piecewise-linear function \(f: \mathbb{R}\to\mathbb{R}\).
\end{definition}
At the non-differentiable points of  \(f\), if
\[\omega_f(x):=\lim_{\varepsilon\to 0_{+}}f'(x + \varepsilon)-f'(x - \varepsilon)<0,
\]
the point $x$ is said to be a \textit{pole} of \(f\) with multiplicity \(-\omega_f(x)\), while if
\(
\omega_f(x)>0,
\)
then \(x\) is called a \textit{root} of \(f\) with multiplicity \(\omega_f(x)\).\par

\begin{definition}\cite{6}\label{D2.2} A tropical entire function $f$ is a tropical meromorphic function with no poles, that is,
$$f(x)=\bigoplus_{j=0}^{+\infty}(a_{j}+ n_{j}x),$$
where all $a_{j}\in\mathbb{T}$ and $n_{j}\in\mathbb{R}.$
\end{definition}

Korhonen and Tohge [\cite{6}, Proposition 3.3] proved that any tropical meromorphic function \(f\) can be expressed as \(f = h\oslash g\), where \(g\) and \(h\)  are two tropical entire functions without any common roots.\par

For a tropical meromorphic function \(f\), its \textit{tropical proximity function} is defined by
\[
m(r,f):=\frac{1}{2}\sum_{\sigma=\pm1}f^{+}(\sigma r)=\frac{f^{+}(r)+f^{+}(-r)}{2},
\] where \(f^{+}(x)=\max\{f(x),0\}\), and the \textit{tropical counting function} is defined by
\[
N(r,f):=\frac{1}{2}\int_{0}^{r}n(t,f)dt=\frac{1}{2}\sum_{|b_v|<r}|\omega_f(b_v)|(r - |b_{\nu}|),
\] where
\[
n(t,f)=\sum_{\substack{|s|\leq t\\ \omega_f(s)<0}}|\omega_f(s)|.
\] The \textit{tropical Nevanlinna's characteristic function} $T(r,f)$ of $f$ is defined as
\[
T(r,f):=m(r,f)+N(r,f).
\]

These tropical Nevanlinna's functions have the following basic properties.\par

\begin{lemma}\cite[Lemma 3.2]{Risto15}\label{L2.3} Let $f$ and $g$ be two tropical meromorphic functions.
\begin{enumerate}
    \item[(i)]
    \[
    \begin{aligned}
    m(r,f \otimes g) &\leq m(r,f) + m(r,g), \\
    N(r,f \otimes g) &\leq N(r,f) + N(r,g), \\
    T(r,f \otimes g) &\leq T(r,f) + T(r,g),\\
    m(r,f \oplus g) &\leq m(r,f) + m(r,g), \\
    N(r,f \oplus g) &\leq N(r,f) + N(r,g), \\
    T(r,f \oplus g) &\leq T(r,f) + T(r,g).
    \end{aligned}
    \]

    \item[(ii)] If \( f \leq g \), then \( m(r,f) \leq m(r,g) \).
    \item[(iii)] Given a positive real number \( \alpha \),
    \[
    \begin{aligned}
    m(r,f^{\otimes \alpha}) &= \alpha m(r,f), \\
    N(r,f^{\otimes \alpha}) &= \alpha N(r,f), \\
    T(r,f^{\otimes \alpha}) &= \alpha T(r,f).
    \end{aligned}
    \]

\end{enumerate}
\end{lemma}

Haulburd and Southall proved the tropical Poisson-Jensen formula [\cite{1}, Lemma 3.1] which implies tropical Jensen formula as follows:
\[
N(r, 1_{\mathbb{T}}\oslash f) - N(r, f) = \frac{1}{2} \sum_{\sigma = \pm 1} f(\sigma r) - f(0).
\]
This implies that
\[
T(r,f) = T(r,1_{\mathbb{T}}\oslash f) + f(0).
\]

\subsection{Tropical holomorphic curves}
Denote by $\mathbb{TP}^n:=\mathbb{T}^{n + 1}\setminus\{(0_{\mathbb{T}}, \dots, 0_{\mathbb{T}})\}$ the tropical projective space, where \[(a_0,a_1,\dots,a_n)\sim(b_0,b_1,\dots,b_n)
\]
if and only if
\[
(a_0,a_1,\dots,a_n)=\lambda\otimes(b_0,b_1,\dots,b_n)=(\lambda + b_0,\lambda + b_1,\dots,\lambda + b_n)
\]
for some \(\lambda\in\mathbb{R}\).  Denote by \([a_0:a_1:\dots:a_n]\) the equivalence class of \((a_0,a_1,\dots,a_n)\). For instance, \(\mathbb{TP}^1\) is just the completed max-plus semiring \(\mathbb{T}\cup\{+\infty\}=\mathbb{R}\cup\{\pm\infty\}\) by the map such that
\[
[1_\mathbb{T}:a]\mapsto a\oslash 1_\mathbb{T} = a \text{ for }a\in\mathbb{T},
\]
and
\[
[0_\mathbb{T}:a]\mapsto a\oslash 0_\mathbb{T} = +\infty\text{ for }a\in\mathbb{R}.
\]

Let $f_0, f_1, \dots, f_n$ be tropical entire functions without any common roots. The tuple $\mathbf{f} = (f_0, \dots, f_n)$ defines a tropical holomorphic map (curve) $f:= [f_0 : \dots: f_n]: \mathbb{R} \to \mathbb{TP}^n,$ and $\mathbf{f}$ is called a reduced representation of $f.$ The \textit{tropical Cartan's characteristic function} of $f$ is defined by
\[
T_f(r):=\frac{1}{2}\sum_{\sigma = \pm 1}\|f(\sigma r)\|-\|f(0)\|=\frac{1}{2}[\|f(r)\|+\|f(-r)\|]-\|f(0)\|
\]
where $\|f(x)\|=\max\{f_0(x),\dots,f_n(x)\}.$ It's well-known that $T_f(r)$ is independent of the reduced representation of $f$ [see \cite[Proposition 4.3]{6}. When $n=1,$ we know that $T(r,f)=T_{f}(r)+O(1).$\par

We now recall the definition of the tropical determinant and some relative operators. The operations of tropical addition $\oplus$ and tropical multiplication $\otimes$ for the $(n + 1)\times(n + 1)$ matrices $A=(a_{ij})$ and $B=(b_{ij})$ are defined by
\[
A\oplus B=(a_{ij}\oplus b_{ij})
\]
and
\[
A\otimes B=\left(\bigoplus_{k = 0}^{n}a_{ik}\otimes b_{kj}\right),
\]
respectively.

The \textit{tropical determinant} $|A|$ of $A$ is defined by
\begin{equation*}
|A|=\bigoplus_{\pi\in S_{n+1}}(a_{0,\pi(0)}\dots a_{n,\pi(n)}),
\end{equation*}
where $S_{n+1}$ is the group of all permutations on $\{0,1,\dots,n\}$.
As in the familiar way, the \textit{transpose} of $A=(a_{i,j})$ is defined by $A^t=(a_{j,i})$. The \textit{minor} $A_{i,j}$ is the matrix obtained by deleting the $i$th row and $j$th column of $A$. The \textit{adjoint} $A^{adj}$ of $A=(a_{i,j})$  is defined as the transpose of the matrix whose entries are the determinants of the corresponding minors, that is, $[A^{adj}]_{i,j}= |A_{j,i}|$.
Equivalently, the tropical determinant $|A|$ can be also  written  in terms of minors as
\begin{equation*}
|A|=\bigoplus_{j}a_{i,j}|A_{i,j}|
\end{equation*}
for some fixed index $i$. The tropical determinant has the following properties \cite[Remark 2.1]{12}:\par
\begin{enumerate}
\item[\textit{(i)}] Transposition and reordering of rows or columns leave the determinant unchanged;
\item[\textit{(iii)}] The determinant is linear with respect to scalar multiplication of any given row or column by a real number.
\end{enumerate}

Choose $c\in\mathbb{R}\setminus\{0\}$. Let $f=(f_0,f_1,\dots,f_n):\mathbb{R}\to\mathbb{TP}^n$ be a tropical holomorphic map. We use short notation
\[
\overline{f}_{j}^{[0]}:=f_j(x),\quad \overline{f}_{j}^{[1]}:=f_j(x + c),\quad \overline{f}_{j}^{[k]}:=f_j(x + kc)
\]
for all $j,k\in\{0,1,\dots,n\}$. The tropical Casorati determinant of $f$ is defined by
\[
C(f):=C(f_0,f_1,\dots,f_n)=\bigoplus_{\pi\in S_{n+1}}\overline{f}_{0}^{[\pi(0)]}\otimes\overline{f}_{1}^{[\pi(1)]}\otimes\dots\otimes\overline{f}_{n}^{[\pi(n)]},
\]
where $S_{n+1}$ is the set of all the permutations on $\{0,1,\dots,n\}$. It has the following properties.\par

\begin{lemma}\cite[Lemma 5.1]{6}
\label{LL2.3}
If \( f_0, \dots, f_n \) and \( h \) are tropical entire functions, then
\begin{enumerate}
    \item[\textit{(i)}] \( C(f_0, f_1, \dots, f_i, \dots, f_j, \dots, f_n) = C(f_0, f_1, \dots, f_j, \dots, f_i, \dots, f_n) \) for all $i,$ $j \in \{0, \dots, n\}$ such that \( i \neq j \).
    \item[\textit{(ii)}] \( C(1_\mathbb{T}, f_1, \dots, f_n) \geq C(\overline{f}_1, \dots, \overline{f}_n) \).
    \item[\textit{(iii)}] \( C(0_\mathbb{T}, f_1, \dots, f_n) = 0_\mathbb{T} \).
    \item[\textit{(iv)}] \( C(f_0 \otimes h, f_1 \otimes h, \dots, f_n \otimes h) = h \otimes \overline{h} \otimes \dots \otimes \overline{h}^{[n]} \otimes C(f_0, f_1, \dots, f_n) \).
\end{enumerate}
\end{lemma}

Assume that $f:=[f_0, \dots, f_{n}]$ is a tropical holomorphic curve into $\mathbb{TP}^{n}.$ Let $\hat{g}$ and $\tilde{g}$ are two tropical linear combinations of the $n + 1$ tropical entire functions $f_0, \dots, f_n$ without common roots. Korhonen and Tohge proved that the tropical Nevanlinna's characteristic function of the quotient of the two tropical entire functions $\hat{g}$ and $\tilde{g}$ can be controlled by the tropical Cartan's characteristic function of $f.$\par

\begin{lemma}\cite[Lemma 4.7]{6} \label{L2.6}If $f = [f_0: \dots: f_n]: \mathbb{R}\to\mathbb{TP}^n$ is a tropical holomorphic curve, then
\[
T\left(r, \hat{g}\oslash\tilde{g}\right)+N_{\min}\left(r, 0, \hat{g}, \tilde{g}\right) \leq T_f(r)+C+\max\{f_0(0),\dots,f_n(0)\}-\tilde{g}(0),
\]
where $C$ is the maximum of the coefficients of the two tropical linear combinations $\hat{g}$ and $\tilde{g}$ over $\mathbb{T}$.
\end{lemma}

The following is the tropical version of the logarithmic derivative lemma with subnormal growth (also known as minimal hypertype).\par

\begin{lemma}\cite[Theorem 3.1]{7}\label{T2.7}
Let $c\in\mathbb{R}\setminus\{0\}$. If $f$ is a tropical meromorphic function on $\mathbb{R}$ with
\begin{equation*}
\limsup_{r\rightarrow\infty}\frac{\log T_{f}(r)}{r}=0,
\end{equation*}
then
\[m(r,f(x + c)\oslash f(x)) = o(T_{f}(r))\]
where $r$ runs to infinity outside of a set of zero upper density measure $E$, i.e.,
\[\overline{dens}E=\limsup_{r\rightarrow\infty}\frac{1}{r}\int_{E\cap[1,r]}dt = 0.\]
\end{lemma}

\begin{lemma}\cite[Lemma 3.3]{7}\label{L2.8} Let $T(r)$ be a nondecreasing, positive, convex, continuous function on $[1,+\infty)$ with
\[
\limsup_{r\to\infty}\frac{\log T(r)}{r}=0.
\]
Then for any fixed positive real value $c,$
\[
T(r)\leq T(r + c)\leq(1 + o(r))T(r),\ r\notin E\to\infty,
\]
where the exceptional set $E$ is a set with zero upper density measure.
\end{lemma}

\section{Tropically linearly independent and tropical Cramer theorem}
\subsection{Tropical linear algebra}

First, we recall the notion of tropical linear dependence for vectors.\par

\begin{definition}\cite[Definition 1.2]{12}\label{D3.1} Let $v_i=(v_{i0}, \dots, v_{in})^{t},$ where $i\in \{1, \dots, m\}.$ A collection of vectors $v_1,$ $\dots,$ $v_m$ is said to be tropically linearly dependent if there exist $\alpha_1, \dots, \alpha_m\in\mathbb{T}$, not all of them equal to \( 0_{\mathbb{T}} \), such that for every $0\leq j\leq n$, the maximum in the expression
\[
\bigoplus_{i=1}^{m}\alpha_i\otimes v_{ij},
\]
is attained at least twice. Otherwise, the collection of vectors $v_1,$ $\dots,$ $v_m$ is said to be tropically linearly independent.
\end{definition}

Inspired by Definition \ref{D3.1}, we introduce the notion of tropical linear dependence (or independence) for tropical meromorphic functions as follows.\par

\begin{definition}\label{D2.3} Tropical meromorphic functions $f_0,\dots,f_n$ are tropically linearly dependent if there exist  $\alpha_0, \dots, \alpha_n\in\mathbb{T}$, not all of them \( 0_{\mathbb{T}} \), such that the maximum in the expression
\[
\bigoplus_{i\in \{0,\dots, n\}}\alpha_i\otimes f_i,
\]
is attained at least twice for each $x\in\mathbb{R}$. Otherwise, the tropical meromorphic functions are said to be tropically linearly independent.
\end{definition}

Interestingly, we observe that our concept of tropical linear dependence in Definition \ref{D3.1} is different from the linearly dependence in the Gondran-Minoux sense due to Korhonen-Tohge \cite{6}: functions $(f_0,\dots,f_n)$ are said to be linearly dependent if there exist two disjoint subsets \( I \) and \( J \) of \( K := \{0, \dots, n\} \) such that \( I \cup J = K \) and
\[
\bigoplus_{i \in I} a_i \otimes f_i = \bigoplus_{j \in J} a_j \otimes f_j,
\]
that is,
\[
\max_{i \in I} \{a_i + f_i\} = \max_{j \in J} \{a_j + f_j\},
\]
where the constants \( a_0, a_1, \dots, a_n \in \mathbb{T} \) are not all equal to \( 0_{\mathbb{T}} \). Clearly, for two tropical meromoprhic functions, the two definitions are equivalent. It is easy to see that if tropical meromorphic functions $f_0, \dots, f_n$ are linearly dependent in the Gondran-Minoux sense, then they are tropically linearly dependent in our sense. We give an example for (more than two) tropical meromorphic functions to illustrate that the inverse relationship may not be true.\par

\begin{figure}
    \centering
    \subfloat{\label{Figure1}}
    \includegraphics[width=0.45\linewidth]{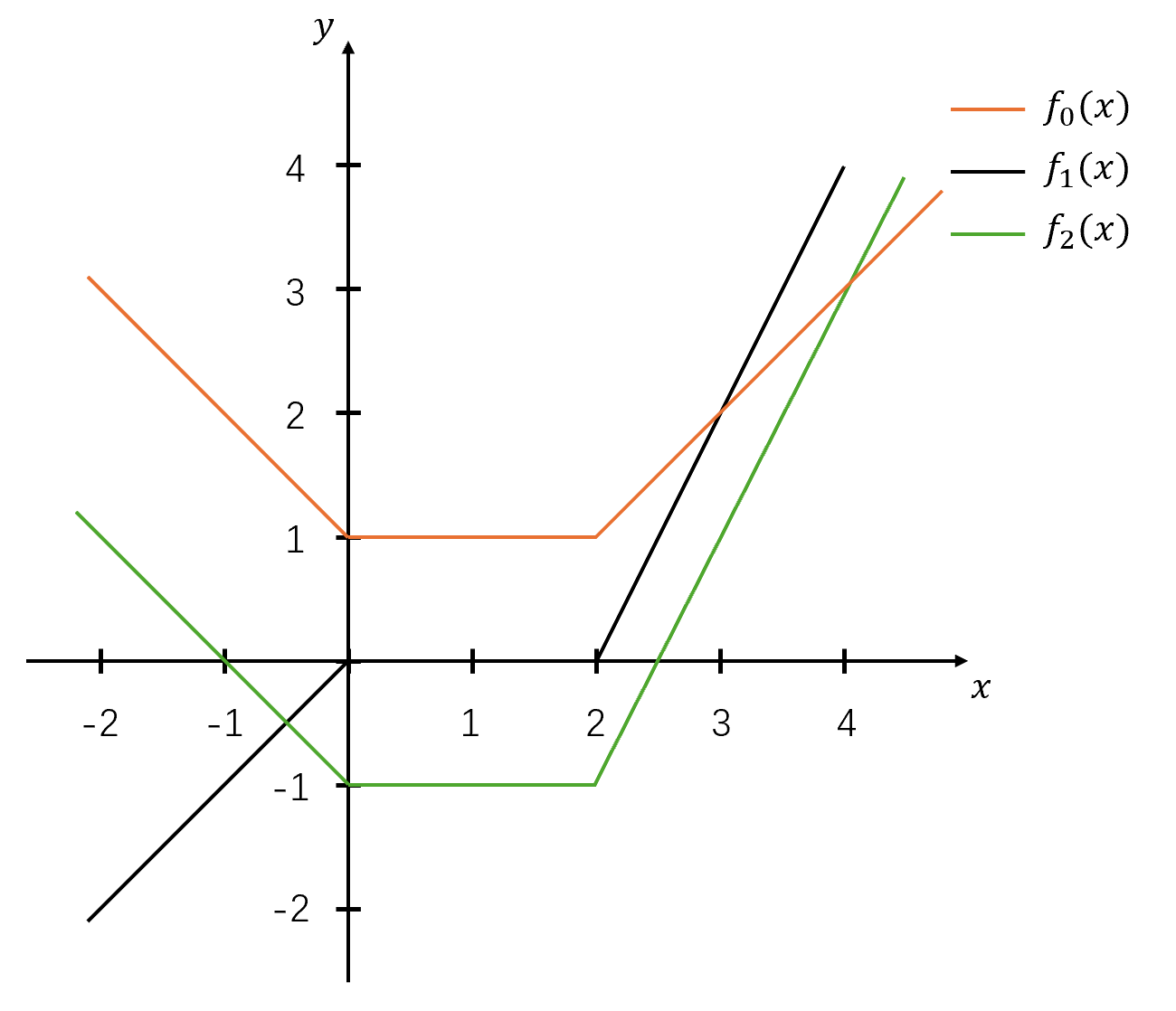}
    \quad
    \subfloat{\label{Figure2}}
    \includegraphics[width=0.45\linewidth]{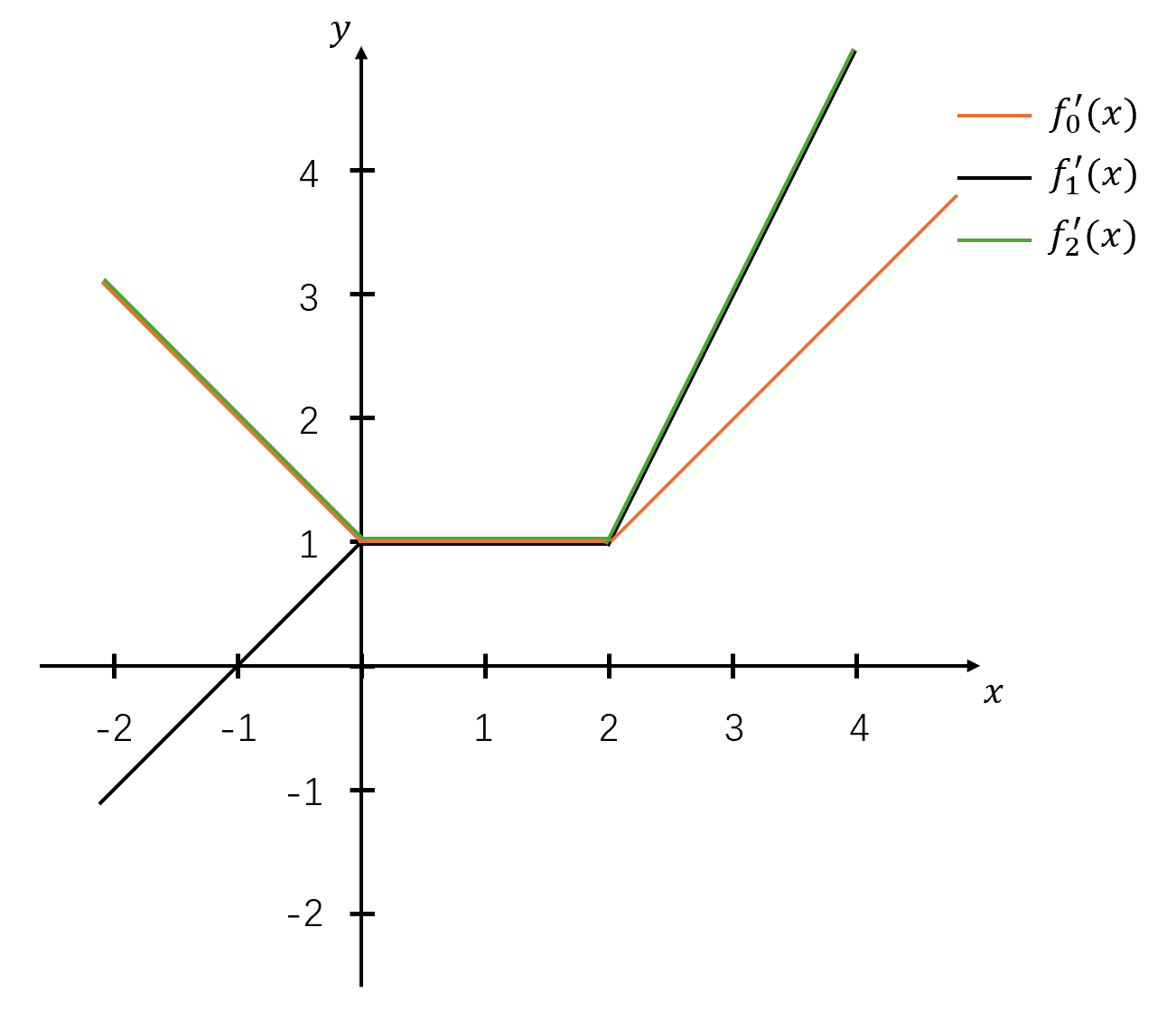}
    \caption{Tropical meromorphic functions in Example \ref{E2.4}}
    \label{fig:image}
\end{figure}

\begin{example}\label{E2.4} Take three tropical meromorphic functions $$f_0(x)=\max\{-x+1, 1, x-1\},$$
\[
f_1(x)=
\begin{cases}
x, & x < 0, \\
0, & 0\leq x\leq 2, \\
2x-4, & x>2,
\end{cases}
\]
and \[
f_2(x)=
\begin{cases}
-x-1, & x < 0, \\
-1, & 0\leq x\leq 2, \\
2x-5, & x>2.
\end{cases}
\]
 From the left hand of Figure \ref{fig:image}, it is clear that none of them can be expressed in terms of the other two functions, so that they are linearly independent in the Gondran-Minoux sense. Denote $\alpha_0:=0, \alpha_1:=1, \alpha_2:=2$ and  $\alpha_i\otimes f_i:=f'_i.$  From the right hand of Figure \ref{fig:image}, one can see that whenever $x < 0$, $f'_0(x)$ and $f'_2(x)$ are equal and greater than $f'_1(x)$; whenever $0\leq x\leq 2$, $f'_0(x)$, $f'_1(x)$ and $f'_2(x)$ are equal; whenever $x > 2$, $f'_1(x)$ and $f'_2(x)$ are equal and greater than $f'_0(x)$.
This yields that the maximum in the expression $$\bigoplus_{i=0}^{2}\alpha_i\otimes f_i$$
is attained at least twice for all values of $x$. Hence, the three functions are tropically linearly dependent.
\end{example}

\begin{definition}\cite[Definition 2.2]{12}\label{D3.2} An $m\times m$-matrix $A\in M_{m\times m}(\mathbb{\mathbb{T}})$ is said to be \textit{tropically singular} if the maximum in the tropical determinant $|A|$ is attained at least twice, otherwise $A$ is called \textit{tropically nonsingular}.
\end{definition}

\begin{lemma}\cite[Lemma 5.1]{RGST-05}\label{L3.3} The matrix $A$ is tropically singular if and only if the $n+1$ points whose coordinates are the column vectors of $A$ lie on a tropical hyperplane in $\mathbb{TP}^{n}$.
\end{lemma}

From the above, we can obtain the following lemma.\par

\begin{lemma}\label{L3.4} The matrix $A$ is tropically singular if and only if the column vectors of $A$ are tropically linearly dependent.
\end{lemma}
\begin{proof} Assume that the matrix $A$ is tropically singular, then clearly, $A^t=(a_{j,i})$ is also tropically singular. By Lemma \ref{L3.3}, there exists a hyperplane $$H= \left\{x \in \mathbb{TP}^n \left| \,\,\max_{i=0}^{n}(\alpha_i \otimes x_i) \text{ is attained at least twice}\right.\right\},$$ such that the column vectors of $A^t=(a_{j,i})$ lie on the tropical hyperplane, i.e. for every $0\leq j\leq n$, the maximum in the expression
\[
\bigoplus_{i=0}^n\alpha_i\otimes a_{ji},
\]
is attained at least twice. Hence, by Definition
\ref{D3.1}, the column vectors of $A$ are tropically linearly dependent.\par

On the contrary, assume that the column vectors of $A$ are tropically linearly dependent, then there exist $\alpha_0,\dots,\alpha_n\in\mathbb{T}$, not all of them \( 0_{\mathbb{T}} \), such that for every $0\leq j\leq n$, the maximum in the expression
\[
\bigoplus_{i=0}^{n}\alpha_i\otimes a_{ji},
\]
is attained at least twice. Hence, the $n+1$ points whose coordinates are the row vectors $(a_{j0},\dots, a_{jn})$ (where $0\leq j\leq n)$ of $A$ lie on the hyperplane $$H= \left\{x \in \mathbb{TP}^n \left| \,\,\max_{i=0}^{n}(\alpha_i + x_i) \text{ is attained at least twice}\right.\right\}.$$  Hence, $A^t=(a_{j,i}),$ is tropically singular, and thus $A$ is also tropically singular.
\end{proof}

By Lemma \ref{L3.3}, Lemma \ref{L3.4}, we can obtain the following conclusion.\par

\begin{lemma}\label{C2.13} Any $q$ tropical hyperplanes $H_1, \dots, H_q$ in $\mathbb{TP}^{n}$ are in general position if and only if their coefficient vectors $a_{\mu(0)}, \dots, a_{\mu(n)}$ are tropically linearly independent for any injective map $\mu : \{0, 1, \dots, n\} \to \{1, \dots, q\}$.
\end{lemma}

\subsection{Tropical Cramer theorem}

In this section, we will provide an upper bound for solutions of tropical linear equations. Additionally, we will present a relationship between determinants in tropical matrix multiplication. These results will be used to prove our tropical second main theorem with tropical hyperplanes in subsection 4.1. First, we introduce Cramer's theorem for tropical linear equations, which is very useful for our results  (for details, see references \cite{Aki-09,RGST-05}).\par

Cramer's theorem for tropical linear equations is shown as follows. We denote by $B_i$ the $i$th Cramer matrix of $(A, b)$, obtained
by replacing the $i$th column of $A$ with $b$. The $i$th Cramer permanent is defined as $|B_i|$. \par

\begin{lemma}\cite[Corollary 6.9]{Aki-09}\label{C3.5} Let $b, x\in\mathbb{T}^{n+1}$ and $A=(a_{ij})\in\mathcal{M}_{n+1 \times n+1}(\mathbb{T})$. Assume that for every row index $0\leq i\leq n$, the maximum in the expression
\begin{equation}\label{eq2.3}
\bigoplus_{j=0}^n(a_{ij}\otimes x_{j}\oplus b_{i})
\end{equation}
is attained at least twice. Then, for all $0\leq i\leq n$, if we expand $|B_i|$ and $|A|$ in
\begin{equation}\label{eq2.4}
|A|\otimes x_{i}\oplus |B_i|,\end{equation}
the maximum is attained at least twice in the global expression. Moreover, if $A$ is tropically nonsingular and if every Cramer matrix $B_{i}$ is tropically nonsingular or $|B_i|=0_{\mathbb{T}},$ then $\hat{x}:=(|B_i|-|A|)_{0\leq i\leq n}$ is the unique vector $x\in\mathbb{T}^{n+1}$ such that the maximum in Expression \eqref{eq2.3} is attained at least twice, for every $0\leq i\leq n$.
\end{lemma}

Clearly, every family of $n+2$ vectors of $\mathbb{T}^{n+1}$ is tropically linearly dependent.\par

When A is tropically nonsingular, by Lemma \ref{C3.5}, we can obtain an upper bound for the solution, as follows.\par

\begin{theorem}\label{The2.20}[Cramer system's upper bound] Let $b, x\in\mathbb{T}^{n+1}$ and $A=(a_{ij})\in\mathcal{M}_{(n+1) \times (n+1)}(\mathbb{T}).$  Assume that for every row index $0\leq i\leq n$, the maximum in the expression
\begin{equation}\label{eq6}
\bigoplus_{j=0}^n(a_{ij}\otimes x_{j}\oplus b_{i})
\end{equation}
is attained at least twice. If A is tropically nonsingular, then for the solution \(x=(x_0, \dots, x_n)\) satisfying equation \eqref{eq6}, we have $x_i\leq (A^{\mathrm{adj}}b)_i-|A|$ for all \(0\leq i\leq n\).
\end{theorem}
\begin{proof}
By Lemma \ref {C3.5}, we can obtain that the maximum value of $|A|\otimes x_{i}\oplus |B_i|$ is attained at least twice for all $0\leq i\leq n$. Next, we will discuss the values of the component coordinates of the solution for the cases where the Cramer matrix $B_{i}$ is tropically singular and tropically nonsingular, respectively.\par

Case 1: if the Cramer matrix $B_{i}$ is tropically nonsingular or $|B_i|=0_{\mathbb{T}},$ then since A is tropically nonsingular, by Lemma \ref {C3.5}, we have $x_i:=(|B_i|-|A|).$\par

Case 2: if the Cramer matrix $B_{i}$ is tropically singular, then by Definition \ref{D3.2}, we have $|B_i|$'s maximum is attained at least twice. Since A is tropically nonsingular, it is easy to find that for any $x_i\in\mathbb{T}$, as long as $x_i\leq |B_i|-|A|$, the maximum value of $|A|\otimes x_{i}\oplus |B_i|$ is attained at least twice.\par

Hence, every component of the solution \(x=(x_0, \dots, x_n)\) satisfying equation \eqref{eq6} must have $x_i\leq |B_i|-|A|$ for all \(0\leq i\leq n\).

Below we only need to verify that $(A^{\mathrm{adj}}b)_i=|B_i|.$ Indeed, we have
\[|B_i|=
\begin{vmatrix}
a_{00} & a_{01} & \dots &b_0  & \dots& a_{0n} \\
a_{10} & a_{11} & \dots &b_1  & \dots& a_{1n} \\
\vdots & \vdots & \ddots & \vdots & \ddots& \vdots \\
a_{n0} & a_{n1} & \dots &b_n  & \dots& a_{nn}
\end{vmatrix}=b_0\otimes A_{0i}\oplus b_1\otimes A_{1i}\oplus\dots \oplus b_n\otimes A_{ni}
\] and thus
\[A^{\mathrm{adj}}b=
\begin{bmatrix}
A_{00} & A_{10} & \dots & A_{n0} \\
A_{01} & A_{11} & \dots & A_{n1} \\
\vdots & \vdots & \ddots & \vdots \\
A_{0n} & A_{1n} & \dots & A_{nn}
\end{bmatrix}\otimes
\begin{bmatrix}
b_0  \\
b_1  \\
\vdots\\
b_n
\end{bmatrix}=\begin{bmatrix}
|B_0|  \\
|B_1|  \\
\vdots\\
|B_n|
\end{bmatrix}.
\] The theorem is thus proved.
\end{proof}

We give an example to explain the above result.\par

\begin{example}\label{E3.9} Consider the following tropical linear equation, where we require that the maximum value of
\[\bigoplus
\begin{bmatrix}
0 & -1 & 1 \\
0 & 0  & 2 \\
0 & 1  & 0
\end{bmatrix}
\begin{bmatrix}
x_1  \\
x_2  \\
x_3
\end{bmatrix}\bigoplus
\begin{bmatrix}
0  \\
-1 \\
1
\end{bmatrix}
\]
is attained at least twice. Here, $|A|=3$ (A is tropically nonsingular),
\[|B_1|=
\begin{vmatrix}
0 & -1 & 1 \\
-1& 0  & 2 \\
1 & 1  & 0
\end{vmatrix}=0\oplus3\oplus(-2)\oplus2\oplus1\oplus2=3
\]($B_1$ is tropically nonsingular), and
\[|B_2|=
\begin{vmatrix}
0 & 0  & 1 \\
0 & -1 & 2 \\
0 & 1  & 0
\end{vmatrix}=3,
|B_3|=
\begin{vmatrix}
0 & -1 & 0 \\
0 & 0  & -1 \\
0 & 1  & 1
\end{vmatrix}=1\oplus0\oplus0\oplus(-2)\oplus1\oplus0=1
\]($B_2$ is tropically nonsingular, $B_3$ is tropically singular). We can obtain the solution to the tropical linear equation as $x_1=0, x_2=0, x_3\leq -2,$ and it is obvious that it satisfies $x_i\leq |B_i| - |A| =(A^{\mathrm{adj}}b)_i-|A|$ for all \(0\leq i\leq 2\).
\end{example}

In linear algebra, a key result states that for square matrices A and B of the same order, $\det(AB)=\det(A)\det(B)$. However, if A and B are two tropical square matrices, there is generally no such simple and direct relationship between the tropical determinant of their product \(C = A\otimes B\) and the tropical determinants of A and B as in the case of conventional matrix multiplication. Interestingly, we have found that there is still an inequality relationship. We will present this relationship.\par

\begin{theorem}\label{T3.12}[The Relationship of Determinants in tropical matrix product] Let $A=(a_{ij}), B=(b_{ij}), C=(c_{ij})$ be $n\times n$ matrices with entries in $\mathbb{T}$, such that $A\otimes B = C$, then $|A|\otimes|B|\leq |A\otimes B|= |C|$.
\end{theorem}

\begin{proof}
We expand the formula $A\otimes B = C$ as
\[
\begin{bmatrix}
a_{11} & a_{12} & \dots & a_{1n} \\
a_{21} & a_{22} & \dots & a_{2n} \\
\vdots & \vdots & \ddots & \vdots \\
a_{n1} & a_{n2} & \dots & a_{nn}
\end{bmatrix}\bigotimes
\begin{bmatrix}
b_{11} & b_{12} & \dots & b_{1n} \\
b_{21} & b_{22} & \dots & b_{2n} \\
\vdots & \vdots & \ddots & \vdots \\
b_{n1} & b_{n2} & \dots & b_{nn}
\end{bmatrix}=
\]
\[
\begin{bmatrix}
\oplus_{i=1}^n(a_{1i}\otimes b_{i1}) & \oplus_{i=1}^n(a_{1i}\otimes b_{i2}) & \dots & \oplus_{i=1}^n(a_{1i}\otimes b_{in}) \\
\oplus_{i=1}^n(a_{2i}\otimes b_{i1}) & \oplus_{i=1}^n(a_{2i}\otimes b_{i2}) & \dots & \oplus_{i=1}^n(a_{2i}\otimes b_{in}) \\
\vdots & \vdots & \ddots & \vdots \\
\oplus_{i=1}^n(a_{ni}\otimes b_{i1}) & \oplus_{i=1}^n(a_{ni}\otimes b_{i2}) & \dots & \oplus_{i=1}^n(a_{ni}\otimes b_{in})
\end{bmatrix}.
\]

Assume that $|A|$ and $|B|$ reach their maximum value at $(a_{1,i(1)}\otimes\dots \otimes a_{n,i(n)})$ and $(b_{j(1),1}\otimes\dots \otimes b_{j(n),n}),$ respectively, where $\{i(1), \dots, i(n)\}$ and $\{j(1), \dots, j(n)\}$ are permutations of $\{1,\dots,n\}.$\par

We define
\[
k_l := j^{-1}(i(l)) \quad \text{for } l=1,2,\dots,n.
\]
Since $\{j(1), \dots, j(n)\}$ is a permutation of $\{1,\dots,n\}$, $\{k_l, \dots, k_n\}$ is also a permutation of $\{1,\dots,n\}$. Then for each $l$, we have
\[
a_{l,i(l)}\otimes b_{j(k_l),k_l}
\leq \bigoplus_{i=1}^n (a_{l,i}\otimes b_{i,k_l})
= c_{l,k_l}.
\]

It follows that
\[
|A|\otimes|B|
= \bigotimes_{l=1}^n a_{l,i(l)} \otimes \bigotimes_{l=1}^n b_{j(l),l}
= \bigotimes_{l=1}^n \left(a_{l,i(l)}\otimes b_{j(k_l),k_l}\right)
\leq \bigotimes_{l=1}^n c_{l,k_l}.
\]
Since $|C|$ is the maximum tropical product over all permutations, we obtain
\[
\bigotimes_{l=1}^n c_{l,k_l} \leq |C|.
\]
Therefore,
\[
|A|\otimes|B| \leq |C|=|A\otimes B|,
\]
which completes the proof.
\end{proof}

An example is given to show Theorem  \ref{T3.12}.\par

\begin{example}\label{E3.13} Let $A$, $B$, and $C$ be the following tropical matrices:
\[
\begin{bmatrix}
1 & 2 & 0 \\
1 & 0 & 1 \\
0 & 1 & 1
\end{bmatrix}
\bigotimes
\begin{bmatrix}
-1 & 1 & -1 \\
0 & 0 & 1 \\
1 & 2 & 1
\end{bmatrix}=
\begin{bmatrix}
2 & 2 & 3 \\
2 & 3 & 2 \\
2 & 3 & 2
\end{bmatrix}
\]
Then $|A|=4, |B|=3, |C|=8$ and $|A|\otimes|B|<|C|=|A\otimes B|$.
\end{example}

\section{Second main theorem}
In this section, we give the proofs of Theorem \ref{T4.1.10} and Theorem \ref{T4.1.2}. Theorem \ref{T4.1.10} uses the tropical logarithmic derivative lemma with minimal hypertype (Lemma \ref{T2.7}), which is closer to the classical Cartan form. Theorem \ref{T4.1.2} avoids any growth condition at the cost of a stronger ramification term (sum of component counting functions), and holds for all r without exceptional sets.

\subsection{Second main theorem with growth condition}
The main goal of this subsection is to prove the second main theorem for tropical hyperplanes in general position with growth condition (Theorem \ref{T4.1.10}).

Recall that Cao and Zheng \cite{7} gave a definition of linear nondegeneracy for tropical holomorphic curve as follows.

\begin{definition}\cite[Definition 3.1]{7}\label{D4.1.3} Let $f = [f_0:f_1:\dots:f_n]:\mathbb{R}\to\mathbb{TP}^n$ be a tropical holomorphic curve. If for any tropical hyperplane $H$ in $\mathbb{TP}^n$, $f(\mathbb{R})$ is not a subset of $H$, then we say that $f$ is tropically linearly nondegenerate.
\end{definition}

Next, we will present the identical relationship between tropical linear nondegeneracy and tropical linear independence.\par

\begin{proposition}\label{P4.1.5} A tropical holomorphic curve $f:\mathbb{R}\to\mathbb{TP}^n$ with reduced representation $f=(f_0,\dots,f_n)$ is tropically linearly nondegenerate if and only if $f_0,\dots,f_n$ are tropically linearly independent.
\end{proposition}

\begin{proof}
Assume that $f=(f_0,\dots,f_n)$ is tropically linearly nondegenerate, this means that for any tropical hyperplane $H$ in $\mathbb{TP}^n$, $f(\mathbb{R})\not\subset H$. Now if $f_0,\dots,f_n$ are tropically linearly dependent, then there exist $\alpha_0,\dots,\alpha_n\in\mathbb{T}$, not all of them $0_{\mathbb{T}}$, such that the maximum in the expression
\[
\bigoplus_{i=0}^n\alpha_i\otimes f_i,
\]
is attained at least twice. Hence, we have a tropical hyperplane \[H = \{x \in \mathbb{TP}^n \mid \max_{i=0}^n(\alpha_i + x_i) \text{ is attained at least twice}\}\] such that $f(\mathbb{R})\subset H$. We obtain a contradiction. Hence $f_0,\dots,f_n$ must be tropically linearly independent.

Now assume that $f_0,\dots,f_n$ are tropically linearly independent, If $f$ is tropically linearly degenerate, then there exists a hyperplane \[H= \{x \in \mathbb{TP}^n \mid \max_{i=0}^n(\alpha_i + x_i) \text{ is attained at least twice}\}\] in $\mathbb{TP}^n$ such that $f(\mathbb{R})\subset H$. This means that for all $x\in\mathbb{R}$, $\max_{i=0}^n(\alpha_i + f_i(x)) \text{ is attained at least twice}$, hence $(f_0(x),\dots,f_n(x))$ $\in H$.
This contradicts the assumption that $f_0,\dots,f_n$ are tropically linearly independent. Hence $f$ is tropically linearly nondegenerate.
\end{proof}

We now recall the definitions and notations on Nevnalinna theory for tropical holomorphic curves \cite{7}. Let $f$ be a tropical holomorphic curve from $\mathbb{R}$ into tropical projective space $\mathbb{TP}^{n},$ $H$ be a tropical hyperplane in $\mathbb{TP}^{n}$ defined by a tropical linear polynomial $P.$ We use $a$ to denote the coefficient vector for the tropical hyperplane $H.$ The \textit{tropical proximity function} $m_f(r, H)$ of $f$ with respect to $H$ is defined as
$$m_f(r,H)=\frac{1}{2}\sum_{\sigma=\pm1}\lambda_{H}(f(\sigma r)),$$
where $\lambda_{H}(f(\sigma r))$ is the \textit{tropical Weil function} defined by
$$\lambda_{H}(f(x))=\frac{\|f(x)\|\otimes \|a\|}{P\circ f}\oslash.$$
Note that $P\circ f$ is a tropical entire function on $\mathbb{R}$ which thus doesn't have any pole. Hence by the tropical Jensen formula, we have
\[
N(r,\frac{1_{\mathbb{T}}}{P\circ f}\oslash )=\frac{1}{2}\sum_{\sigma=\pm1}P\circ f(\sigma r)-P\circ f(0).
\]
The following first main theorem for a tropical hyperplane can be easily deduced from the definitions of tropical characteristic function, counting function and approximation function.

\begin{theorem}\label{T4.1.6}[First Main Theorem]
If $f(\mathbb{R}) \not\subset H$, then
    \[
    m_f(r, H) + N(r,\frac{1_{\mathbb{T}}}{P\circ f}\oslash ) = T_f(r) + O(1).
    \]
\end{theorem}
Below we give a general second main theorem for tropical holomorphic curves intersecting arbitrary tropical hyperplanes.\par

\begin{theorem}\label{T4.1.7}[A General second main theorem with tropical hyperplanes] Assume that the tropical holomorphic curve \(f = [f_0 : \dots : f_n]: \mathbb{R}\to\mathbb{TP}^n\) is tropically linearly nondegenerate. Let \(H_1,\dots,H_q\) be arbitrary tropical hyperplanes in \(\mathbb{TP}^n\) defined by tropical linear polynomials $P_j$ ($j = 1,\dots,q$), respectively. If $\limsup_{r \to \infty} \frac{\log T_f(r)}{r}=0,$ then
\begin{equation*}\label{EE}
\begin{split}
\sum_{\sigma=\pm1}\max_{K}\frac{1}{2}\sum_{k\in K}\lambda_{H_{k}}(f(\sigma r))\leq(n + 1)T_{f}(r) - \sum_{j = 0}^{n}N(r, \frac{1_{\mathbb{T}}}{C(f)}\oslash) + o(T_f(r)),
\end{split}
\end{equation*}
where $r$ tends to infinity outside of a set of zero upper density measure. Here the maximum is taken over all subsets \(K\) of \(\{1,\dots,q\}\) such that \(a_{j}, j\in K\), are tropically linearly independent.
\end{theorem}

\begin{proof}
Let $H_1,\dots,H_q$ be the given tropical hyperplanes with coefficient vectors $a_1,\dots,a_q$ in $\mathbb{TP}^n$. \(K\subset \{1,\dots,q\}\) be a subset such that \(a_{k}, k\in K\), are tropically linearly independent. Without loss of generality, we assume that \(q\geq n + 1\) and \(\#K=n + 1\). Let \(T\) be the set of all injective maps \(\mu:\{0,1,\dots,n\}\to\{1,\dots,q\}\) such that \(a_{\mu(0)},\dots,a_{\mu(n)}\) are tropically linearly independent. Then
\begin{equation}\label{eq7}
\begin{split}
&\sum_{\sigma=\pm1}\max_{K}\sum_{k\in K}\lambda_{H_{k}}(f(\sigma r))=\sum_{\sigma=\pm1}\max_{\mu \in T}\sum_{j = 0}^{n}\frac{\|f(\sigma r)\|\otimes \|a_{\mu(j)}\|}{P_{\mu(j)}\circ f(\sigma r)}\oslash\\
\leq&\sum_{\sigma=\pm1}\max_{\mu \in T}\sum_{j = 0}^{n}\left(\frac{\|f(\sigma r)\|}{P_{\mu(j)}\circ f(\sigma r)}\oslash\right)+\sum_{\sigma=\pm1}\max_{\mu \in T}\sum_{j = 0}^{n}\|a_{\mu(j)}\|\\
=&\sum_{\sigma=\pm1}\max_{\mu \in T}\left(\frac{\|f(\sigma r))\|^{\otimes n + 1}}{\bigotimes_{j = 0}^{n}P_{\mu(j)}\circ f(\sigma r)}\oslash\right)+O(1)\\
=&\sum_{\sigma=\pm1}\max_{\mu \in T}\left(\frac{|C(P_{\mu(0)}\circ f(\sigma r),\dots,P_{\mu(n)}\circ f(\sigma r))|}{\bigotimes_{j = 0}^{n}P_{\mu(j)}\circ f(\sigma r)}\oslash\right.\\
&\left.\bigotimes\frac{\|f(\sigma r))\|^{\otimes n + 1}}{|C(P_{\mu(0)}\circ f(\sigma r),\dots,P_{\mu(n)}\circ f(\sigma r))|}\oslash\right)+O(1)\\
\leq&\sum_{\sigma=\pm1}\max_{\mu \in T}\frac{|C(P_{\mu(0)}\circ f(\sigma r),\dots,P_{\mu(n)}\circ f(\sigma r))|}{\bigotimes_{j = 0}^{n}P_{\mu(j)}\circ f(\sigma r)}\oslash\\
&+\sum_{\sigma=\pm1}\frac{\|f(\sigma r))\|^{\otimes n + 1}}{|C(f_{0},\dots,f_{n})|}\oslash+O(1),
\end{split}
\end{equation}
where $C(P_{\mu(0)}\circ f(x),\dots,P_{\mu(n)}\circ f(x))$ denotes the tropical Casoratian of $P_{\mu(0)}\circ f(x),\dots,P_{\mu(n)}\circ f(x)$.

In the above, apply Theorem \ref{T3.12} to the following expression,
\[
\begin{bmatrix}
a_{\mu(0),0} & a_{\mu(0),1} & \dots & a_{\mu(0),n} \\
a_{\mu(1),0} & a_{\mu(1),1} & \dots & a_{\mu(1),n} \\
\vdots & \vdots & \ddots & \vdots \\
a_{\mu(n),0} & a_{\mu(n),1} & \dots & a_{\mu(n),n}
\end{bmatrix}\bigotimes
\begin{bmatrix}
f_0(x) & f_0(x+c) & \dots & f_0(x+nc) \\
f_1(x) & f_1(x+c) & \dots & f_1(x+nc) \\
\vdots & \vdots & \ddots & \vdots \\
f_n(x) & f_n(x+c) & \dots & f_n(x+nc)
\end{bmatrix}=
\]
\[
\begin{bmatrix}
\oplus_{i=0}^n(a_{\mu(0),i}+f_i(x)) & \oplus_{i=0}^n(a_{\mu(0),i}+f_i(x+c)) & \dots & \oplus_{i=0}^n(a_{\mu(0),i}+f_i(x+nc)) \\
\oplus_{i=0}^n(a_{\mu(1),i}+f_i(x)) & \oplus_{i=0}^n(a_{\mu(1),i}+f_i(x+c)) & \dots & \oplus_{i=0}^n(a_{\mu(1),i}+f_i(x+nc)) \\
\vdots & \vdots & \ddots & \vdots \\
\oplus_{i=0}^n(a_{\mu(n),i}+f_i(x)) & \oplus_{i=0}^n(a_{\mu(n),i}+f_i(x+c)) & \dots & \oplus_{i=0}^n(a_{\mu(n),i}+f_i(x+nc))
\end{bmatrix}.
\]
Then we obtain
$$C(f_{0},\dots,f_{n})\leq C(P_{\mu(0)}\circ f,\dots,P_{\mu(n)}\circ f)+C,$$
where $C$ is a constant.\par

First, we deal with the first term on the right side of \eqref{eq7}. Denote
\[
L_{\mu}:=\frac{C(P_{\mu(0)}\circ f,\dots,P_{\mu(n)}\circ f)}{\bigotimes_{j = 0}^{n}P_{\mu(j)}\circ f}\oslash.
\]
Since $f$ is tropically linearly
nondegenerate, not all $f_j$ are identically equal to $0_{\mathbb{T}}.$ Then by Lemma \ref{LL2.3} (iv), we have
\[
\begin{split}
\tilde{L_{\mu}}:=&\frac{C(P_{\mu(0)}\circ f,\dots,P_{\mu(n)}\circ f)}{f_0 \otimes \overline{f_0} \otimes \dots \otimes \overline{f_0}^{[n]}}\bigotimes\frac{f_0\otimes\dots\otimes f_0}{P_{\mu(0)}\circ f\otimes\dots\otimes P_{\mu(n)}\circ f}\\
=&\frac{C((P_{\mu(0)}\circ f)\oslash f_0,\dots, (P_{\mu(n)}\circ f)\oslash f_0)}{((P_{\mu(0)}\circ f)\oslash f_0)\otimes\dots\otimes ((P_{\mu(n)}\circ f)\oslash f_0))}\oslash.
\end{split}
\]
Hence we have the following relationship
\[
L_{\mu}=\tilde{L_{\mu}}\otimes A,
\] where $$A:=\frac{f_0 \otimes \overline{f_0} \otimes \dots \otimes \overline{f_0}^{[n]}}{f_0\otimes\dots\otimes f_0}.$$

For $\sigma=+1$ or $-1,$ there exist two injective maps $\mu^{'}$ and $\mu^{''}$  \(:\{0,1,\dots,n\}\to\{1,\dots,q\}\) satisfying
\[
\begin{split}
L_{\mu^{'}}(r):&=\max_{\mu\in T}\frac{|C(P_{\mu(0)}\circ f(r),\dots,P_{\mu(n)}\circ f(r))|}{\bigotimes_{j = 0}^{n}P_{\mu(j)}\circ f( r)}\oslash\\
&=\frac{|C(P_{\mu^{'}(0)}\circ f(r),\dots,P_{\mu^{'}(n)}\circ f(r))|}{\bigotimes_{j = 0}^{n}P_{\mu^{'}(j)}\circ f(r)}\oslash
\end{split}
\]
and
\[
\begin{split}
L_{\mu^{''}}(-r):&=\max_{\mu\in T}\frac{|C(P_{\mu(0)}\circ f(-r),\dots,P_{\mu(n)}\circ f(-r))|}{\bigotimes_{j = 0}^{n}P_{\mu(j)}\circ f(-r)}\oslash\\
&=\frac{|C(P_{\mu^{''}(0)}\circ f(-r),\dots,P_{\mu^{''}(n)}\circ f(-r))|}{\bigotimes_{j = 0}^{n}P_{\mu^{''}(j)}\circ f(-r)}\oslash,
\end{split}
\] respectively. By Lemma \ref{L2.3} (i),  we have
\begin{eqnarray}\label{eq3.2}
L&:=&\sum_{\sigma=\pm1}\max_{\mu \in T}\frac{|C(P_{\mu(0)}\circ f(\sigma r),\dots,P_{\mu(n)}\circ f(\sigma r))|}{\bigotimes_{j = 0}^{n}P_{\mu(j)}\circ f(\sigma r)}\oslash\nonumber\\&=&L_{\mu^{'}}(r)+L_{\mu^{''}}(-r)\nonumber\\
&=&(\tilde{L_{\mu^{'}}}\otimes A)(r)+(\tilde{L_{\mu^{''}}}\otimes A)(-r)\nonumber\\
&=&\tilde{L_{\mu^{'}}}(r)+\tilde{L_{\mu^{''}}}(-r)+\sum_{\sigma=\pm1}A(\sigma r)\nonumber\\
&\leq& \tilde{L_{\mu^{'}}}(r)\oplus \tilde{L_{\mu^{''}}}(r)+\tilde{L_{\mu^{'}}}(-r)\oplus \tilde{L_{\mu^{''}}}(-r)+\sum_{\sigma=\pm1}A(\sigma r)\nonumber\\
&=&\sum_{\sigma=\pm1}\left(\tilde{L_{\mu^{'}}}(\sigma r)\oplus \tilde{L_{\mu^{''}}}(\sigma r)\right)+\sum_{\sigma=\pm1}A(\sigma r)\nonumber\\
&=&\sum_{\sigma=\pm1}(\tilde{L_{\mu^{'}}}\oplus \tilde{L_{\mu^{''}}})^+(\sigma r)-\sum_{\sigma=\pm1}(-(\tilde{L_{\mu^{'}}}\oplus \tilde{L_{\mu^{''}}}))^+(\sigma r)+\sum_{\sigma=\pm1}A(\sigma r)\nonumber\\
&=&2m(r, \tilde{L_{\mu^{'}}}\oplus \tilde{L_{\mu^{''}}})-2m(r, 1_{\mathbb{T}}\oslash(\tilde{L_{\mu^{'}}}\oplus \tilde{L_{\mu^{''}}}))+\sum_{\sigma=\pm1}A(\sigma r)\nonumber\\
&\leq&2m(r, \tilde{L_{\mu^{'}}}\oplus \tilde{L_{\mu^{''}}})+\sum_{\sigma=\pm1}A(\sigma r)\nonumber\\
&\leq&2m(r, \tilde{L_{\mu^{'}}})+2m(r, \title{L_{\mu^{''}}})+\sum_{\sigma=\pm1}A(\sigma r).
\end{eqnarray}

Next, we estimate $m(r, \tilde{L_{\mu^{'}}})$ from below. Denote by
$$g^{'}_{i}(x):= P_{\mu(i)^{'}}\circ f(x), \quad g^{''}_{i}(x):= P_{\mu(i)^{''}}\circ f(x), \quad 0\leq i\leq n.$$ Since $g^{'}_{i}\oslash f_0$ are tropical meromorphic functions for all $j\in\{0, 1, \dots, n\}$, by Lemma \ref{L2.6}, we have\\
$$T_{g^{'}_{i}\oslash f_0}(r)\leq T_f(r)+ O(1).$$
This implies that
$$\limsup_{r\rightarrow\infty}\frac{\log T_{g^{'}_{i}\oslash f_0}(r)}{r}\leq\limsup_{r\rightarrow\infty}\frac{\log T_f(r)}{r}=0$$
holds for all $r\notin E$ with $\overline{dens}E = 0$. Note that $\tilde{L_{\mu^{'}}}$ consists purely of tropical sums and products of the form $(\overline{g^{'}_j}^{[m]} \oslash\overline{f_0}^{[m]})\oslash ({g^{'}_j}\oslash{f_0})$ where $ m \in \{0, 1, \dots, n\}$ and $j \in \{0, \dots, n\}$. Therefore,
\[
m(r, \tilde{L_{\mu^{'}}}) \leq \sum_{j=1}^n \sum_{m=0}^n m\left(r, (\overline{g^{'}_j}^{[m]} \oslash\overline{f_0}^{[m]})\oslash ({g^{'}_j}\oslash{f_0})\right).
\]
By Theorem \ref{T2.7}, we have
\[
\sum_{j=1}^n \sum_{m=0}^n m\left(r, (\overline{g^{'}_j}^{[m]} \oslash\overline{f_0}^{[m]})\oslash ({g^{'}_j}\oslash{f_0})\right)=o(T_f(r))
\]
for $r\notin E$ with $\overline{dens}E = 0$.
Hence,
\begin{equation}\label{eq3.3}
\begin{split}
m(r, \tilde{L_{\mu^{'}}})\leq o(T_f(r))
\end{split}\end{equation}
holds for all $r\notin E$ with $\overline{dens}E = 0$.

Similarly, we obtain
\begin{equation}\label{eq3.4}
\begin{split}
m(r, \tilde{L_{\mu^{''}}})\leq o(T_f(r))
\end{split}
\end{equation}
for all $r\notin E$ with $\overline{dens}E = 0$.

The next step is to estimate $\sum_{\sigma=\pm1}A(\sigma r)$. By the tropical Jensen formula, we have
\begin{equation}\label{eqq3.5}
\begin{split}
\sum_{\sigma=\pm1}A(\sigma r)&=\sum_{\sigma=\pm1}\left(\frac{f_0 \otimes \overline{f_0} \otimes \dots \otimes \overline{f_0}^{[n]}}{f_0\otimes\dots\otimes f_0}\oslash\right)(\sigma r)\\
&=2\sum_{j=1}^n\left(N(r, 1_{\mathbb{T}} \oslash \overline{f}_0^{[j]})-N(r, 1_{\mathbb{T}} \oslash {f}_0)\right)+O(1).
\end{split}
\end{equation}
Using the tropical Jensen formula again, we have
\[
\begin{split}
N(r, 1_{\mathbb{T}} \oslash f_0) &=\frac{1}{2}\sum_{\sigma = \pm 1}f_0(\sigma r)+O(1)\\
&\leq\frac{1}{2}\sum_{\sigma = \pm 1}\max\{f_0(\sigma r),\dots, f_n(\sigma r)\} +O(1)\\
&=T_f(r)+ O(1).
\end{split}
\]
This implies that
$$\limsup_{r\rightarrow\infty}\frac{\log N(r, 1_{\mathbb{T}}\oslash f_0)}{r}\leq\limsup_{r\rightarrow\infty}\frac{\log T_f(r)}{r}=0.$$ Therefore, by Lemma \ref{L2.8}, we have
\begin{equation}\label{eqq3.6}
\begin{split}
N\left(r, 1_{\mathbb{T}} \oslash \overline{f}_0^{[j]}\right) &\leq N(r + j|c|, 1_{\mathbb{T}} \oslash f_0)\\
&= N(r, 1_{\mathbb{T}} \oslash f_0)+\varepsilon(r)N(r, 1_{\mathbb{T}}\oslash f_0)\\
&= N(r, 1_{\mathbb{T}} \oslash f_0)+ o(T_f(r))
\end{split}
\end{equation}
for $r\notin E$ with $\overline{dens}E = 0$.
Now by \eqref{eqq3.5} and \eqref{eqq3.6}, we have
\begin{eqnarray}\label{Eqqmath}\sum_{\sigma=\pm1}A(\sigma r)\leq o(T_f(r)).\end{eqnarray}\par

Therefore, combining \eqref{eq3.2}, \eqref{eq3.3}, \eqref{Eqqmath} and \eqref{eq3.4}, we can obtain
\begin{equation}\label{eq3.5}
\sum_{\sigma=\pm1}\max_{\mu \in T}\frac{|C(P_{\mu(0)}\circ f(\sigma r),\dots,P_{\mu(n)}\circ f(\sigma r))|}{\bigotimes_{j = 0}^{n}P_{\mu(j)}\circ f(\sigma r)}\oslash\leq o(T_f(r))
\end{equation}
for all $r\notin E$ with $\overline{dens}E = 0.$\par

Finally, we deal with the second term on the right side of \eqref{eq7}. By the tropical Jensen formula again
\begin{equation}\label{eq3.6}
\begin{split}
&\sum_{\sigma=\pm1}\frac{\|f(\sigma r))\|^{\otimes n + 1}}{|C(f_{0},\dots,f_{n})|}\oslash \\
&=\sum_{\sigma=\pm1}\|f(\sigma r))\|^{\otimes n + 1}-\sum_{\sigma=\pm1}{|C(f_{0},\dots,f_{n})|} \\
&=2(n + 1)T_f(r) - 2N(r, \frac{1_{\mathbb{T}}}{|C(f_{0},\dots,f_{n})|}\oslash ) + O(1).
\end{split}
\end{equation}
Combining \eqref{eq7}, \eqref{eq3.5} and \eqref{eq3.6} implies \eqref{EE}. and thus the proof is completed.
\end{proof}

To prove Theorem \ref{T4.1.10}, we also need the following lemma.\par

\begin{lemma}\label{L3.7}[Tropical version of product to the sum estimate] Let $H_1, \dots, H_q$ be tropical hyperplanes in $\mathbb{TP}^n$ defined by tropical linear polynomials $P_j$ ($j = 1,\dots,q$), respectively, located in general position. Denote by $T$ the set of all injective map $\mu:\{0, 1, \dots, n\} \to \{1, \dots, q\}$. Then
$$\sum_{j = 1}^{q} m_f(r, H_j) \leq \sum_{\sigma=\pm1}\max_{\mu\in T}\frac{1}{2}\sum_{j = 0}^{n}\lambda_{H_{\mu_j}}(f(\sigma r)) + O(1).$$
\end{lemma}
\begin{proof}
Suppose that $a_j$ are the coefficient vectors of $H_j, (1\leq j\leq q)$ and then we have tropical entire functions
\[
P_j\circ f=a_{j0}\otimes f_0\oplus\dots\oplus a_{jn}\otimes f_n, 1\leq j\leq q.
\]
By Theorem \ref{The2.20}, we have
\[
f_i\leq \tilde{a}^{\mu(i)} _0\otimes P_{\mu(0)}(f)\oplus\dots\oplus \tilde{a}^{\mu(i)} _n\otimes P_{\mu(n)}(f)-|A_{\mu}|
\]
where $A^{\mathrm{adj}}_{\mu}=(\tilde{a}^{\mu(i)} _j)$. Thus,
for any $\mu\in T$,
\begin{equation*}
\begin{split}
\|f(x)\|=\max\{f_0(x),\dots,f_n(x)\}\leq C+\max_{0\leq i\leq n}\{P_{\mu(i)}\circ f\}.
\end{split}
\end{equation*}
For a given $x\in\mathbb{R}$, there is $\mu\in T$ such that
\[
0_{\mathbb{T}}<P_{\mu(0)}\circ f<\dots<P_{\mu(n)}\circ f<P_j\circ f
\]
for $j\neq\mu(i),i = 0,1,\dots,n$. Hence we have
\[
\bigotimes_{j = 1}^{q}\left(\frac{\|f(x)\|}{P_{j}\circ f}\oslash\right)\leq C+\left(\max_{\mu\in T}\left(\bigotimes_{i = 0}^{n}\frac{\|f(x)\|}{P_{\mu(i)}\circ f}\oslash\right)\right).
\]
The lemma is thus proved.
\end{proof}

\noindent{\bf Proof of Theorem \ref{T4.1.10}.}  Combining Theorem \ref{T4.1.7}, Lemma \ref{L3.7} with the first main theorem (Theorem \ref{T4.1.6}) implies immediately Theorem \ref{T4.1.10}.\qed

\subsection{Second main theorem without growth condition}

In this subsection, we give a second main theorem for tropical hyperplanes in general position in which the tropical holomorphic curve has no growth limitation. \par\vskip 6pt

\noindent {\bf Theorem \ref{T4.1.2}.}  \emph{Assume that the tropical holomorphic curve \(f = [f_0:\dots:f_n]:\mathbb{R}\to\mathbb{TP}^n\) is tropically linearly nondegenerate. If \(H_1,\dots,H_q\) are tropical hyperplanes in \(\mathbb{TP}^n\) in general position defined by tropical linear polynomials $P_j$ ($j = 1,\dots,q$), respectively,
then
\[
(q - n - 1)T_{f}(r)\leq\sum_{j = 1}^{q} N(r, \frac{1_{\mathbb{T}}}{P_j\circ f} \oslash )) -\sum_{j = 0}^{n}N(r, 1_{\mathbb{T}} \oslash f_j)+ O(1).
\]}\vskip 6pt

\begin{proof}
Let $H_1,\dots,H_q$ be the given tropical hyperplanes with coefficient vectors $a_1,\dots,a_q$ in $\mathbb{TP}^n$. Let \(T\) be the set of all injective map \(\mu:\{0, 1 , \dots , n \}\to\{1 , \dots , q \}\). Then
\begin{equation}\label{eq21}
\begin{split}
\sum_{\sigma=\pm1}\max_{\mu\in T}\sum_{j = 0}^{n}&\lambda_{H_{\mu(j)}}(f(\sigma r))\\
=&\sum_{\sigma=\pm1}\max_{\mu \in T}\sum_{j = 0}^{n}\frac{\|f(\sigma r)\|\otimes \|a_{\mu(j)}\|}{P_{\mu(j)}\circ f(\sigma r)}\oslash\\
\leq&\sum_{\sigma=\pm1}\max_{\mu \in T}\sum_{j = 0}^{n}\left(\frac{\|f(\sigma r)\|}{P_{\mu(j)}\circ f(\sigma r)}\oslash\right)+\sum_{\sigma=\pm1}\max_{\mu \in T}\sum_{j = 0}^{n}\|a_{\mu(j)}\|\\
=&\sum_{\sigma=\pm1}\max_{\mu \in T}\left(\frac{\|f(\sigma r))\|^{\otimes n + 1}}{\bigotimes_{j = 0}^{n}P_{\mu(j)}\circ f(\sigma r)}\oslash\right)+O(1).
\end{split}
\end{equation}

We now deal with the right side of \eqref{eq21}. Since \(H_1,\dots, H_q\) are in general position,   $$H_{\mu(0)}\cap \dots\cap H_{\mu(n)}=\varnothing$$ holds for any map \(\mu\in T\). Then by Theorem \ref{C2.13}, \( a_{\mu(0)}, \dots, a_{\mu(n)} \) are tropically linearly independent. Thus the tropical determinant $|A|$'s maximum is attained once, that is, there exists a unique permutation \(\{\delta(0),\dots,\delta(n)\}\) of \(\{0,1,\dots,n\}\), such that \(a_{\mu(0),\delta(0)}\otimes a_{\mu(1),\delta(1)}\otimes\dots\otimes a_{\mu(n),\delta(n)}\) attains the maximum (which is not \( 0_{\mathbb{T}} \)). Clearly, none of \(a_{\mu(0),\delta(0)}, a_{\mu(1),\delta(1)}, \dots, a_{\mu(n),\delta(n)}\) is equal to \( 0_{\mathbb{T}} \).
Hence, for every \( 0\leq j\leq n\), we have
\begin{eqnarray}\label{eq22}
P_{\mu(j)}\circ f&=& a_{\mu(j),0}\otimes f_0\oplus a_{\mu(j),1}\otimes f_1\oplus\dots a_{\mu(j),n}\otimes f_n\nonumber
\\&\geq& a_{\mu(j),\delta(j)}\otimes f_{\delta(j)}.
\end{eqnarray}
Then, for any given $x\in\mathbb{R}$,
\begin{equation}\label{eq23}
\begin{split}
\frac{\|f(x))\|^{\otimes n + 1}}{\bigotimes_{j = 0}^{n}P_{\mu(j)}\circ f(x)}\oslash=& \|f(x))\|^{\otimes n + 1}- \bigotimes_{j = 0}^{n}P_{\mu(j)}\circ f(x)\\
\leq&\|f(x))\|^{\otimes n + 1}- \bigotimes_{j = 0}^{n}a_{\mu(j),\delta(j)}\otimes f_{\delta(j)}(x)\\
=&\|f(x))\|^{\otimes n + 1}- \bigotimes_{j = 0}^{n}f_j(x)+C,
\end{split}
\end{equation} where $C$ is a constant. By the tropical Jensen formula, we have
\begin{equation}\label{eq211}N(r,1_{\mathbb{T}}\oslash f_j)=\frac{1}{2}\sum_{\sigma=\pm1}f_j(\sigma r)-f_j(0).\end{equation}
Due to the arbitrariness of \(\mu\) and $x$, combining \eqref{eq21}, \eqref{eq23} and \eqref{eq211}, we  obtain
\begin{equation}\label{eq24}
\begin{split}
\sum_{\sigma=\pm1}\max_{\mu\in T}\sum_{j = 0}^{n}\lambda_{H_{\mu(j)}}(f(\sigma r))\leq\sum_{\sigma=\pm1}\|f(\sigma r))\|^{\otimes n + 1}-2\sum_{j = 0}^{n}N(r,1_{\mathbb{T}}\oslash f_j)+ O(1).
\end{split}
\end{equation}\par

Next, we  claim that
\begin{equation}\label{eq25}
\begin{split}
\sum_{j = 1}^{q} 2m_f(r, H_j) \leq \sum_{\sigma=\pm1}\max_{\mu\in T}\sum_{j = 0}^{n}\lambda_{H_{\mu_j}}(f(\sigma r)) + O(1).
\end{split}
\end{equation} Indeed, by \eqref{eq22}, for any $\mu\in T$, there exists a constant $C$ such that
\begin{equation*}
\begin{split}
\|f(x)\|=\max\{f_0(x),\dots,f_n(x)\}\leq C+\max_{0\leq i\leq n}\{P_{\mu(i)}\circ f\}.
\end{split}
\end{equation*}
For a given $x\in\mathbb{R}$, there is $\mu\in T$ such that
\[
0_{\mathbb{T}} <P_{\mu(0)}\circ f<\dots<P_{\mu(n)}\circ f<P_{j}\circ f
\]
for $j\neq\mu(i),i = 0,1,\dots,n$. Hence we have
\[\label{eq26}
\begin{split}
\bigotimes_{j = 1}^{q}\left(\frac{\|f(x)\|}{P_{j}\circ f}\oslash\right)\leq C+\left(\max_{\mu\in T}\left(\bigotimes_{i = 0}^{n}\frac{\|f(x)\|}{P_{\mu(i)}\circ f}\oslash\right)\right).
\end{split}
\]
This implies \eqref{eq25} in terms of the definition of proximity function.\par

Therefore, combining \eqref{eq24} and \eqref{eq25}, the proof is completed.
\end{proof}

\noindent{\bf Theorem \ref{Tt3.12}.} \emph{Assume that the tropical holomorphic map $f=[f_0, f_1, \dots, f_n]: \mathbb{R}\to\mathbb{TP}^n$ is tropically linearly nondegenerate. If $H$ is a tropical hyperplane in \( \mathbb{TP}^n \) defined by the tropical polynomial
\begin{eqnarray*} P(x):=(a_0 \otimes x_0)\oplus \dots \oplus ( a_n \otimes x_n)
\end{eqnarray*} such that all $a_0, \dots, a_n \in\mathbb{R},$ then
\[
T_f(r) = N\left( r, \frac{1_{\mathbb{T}}}{P\circ f} \oslash \right) + O(1).
\]}
\par\vskip 6pt

\begin{proof} Take some particular tropical hyperplanes $H_1, \dots, H_{n+2}$ in \( \mathbb{TP}^n \) defined by the tropical polynomials
$P_1(x):=1_{\mathbb{T}}\otimes x_0=x_0, P_{2}(x):=1_{\mathbb{T}}\otimes x_1=x_1,$ $\dots,$ $P_{n+1}(x):=1_{\mathbb{T}}\otimes x_n=x_n,$ and $P_{n+2}(x):= P(x),$ respectively. This implies that $P_{j}(f(x))=f_{j}(x)$ for each $1\leq j\leq n+1$ and $P_{n+2}(f(x))=P(f(x)).$ It is obvious that
\begin{equation}\label{eq111}
H_{1}\cap \dots\cap H_{n+1}=\varnothing.
\end{equation}

We claim that \begin{equation}\label{eq33}
H_{1}\cap \dots\cap H_{i-1}\cap H_{i+1}\cap\dots \cap H_{n+2}=\varnothing.
\end{equation}
for any $0\leq i \leq n+1.$ Indeed, for these tropical hyperplanes $H_{1},$ $\dots,$ $H_{i-1},$ $H_{i+1},$ $\dots,$ $H_{n+2}$, the determinant of the matrix formed by their coefficient vectors is
\[|A_i|:=
\begin{bmatrix}
1_{\mathbb{T}} & 0_{\mathbb{T}} & \dots & 0_{\mathbb{T}} & 0_{\mathbb{T}} \\
0_{\mathbb{T}} & 1_{\mathbb{T}} & \dots & 0_{\mathbb{T}} & 0_{\mathbb{T}}\\
\vdots & \vdots & \ddots & \vdots & \vdots\\
0_{\mathbb{T}} & 0_{\mathbb{T}} & \dots & 0_{\mathbb{T}} & 1_{\mathbb{T}}\\
a_0 & a_1 & \dots & a_{n-1} & a_n
\end{bmatrix}=a_{i-1}.
\] Since all $a_0, \dots, a_n \in\mathbb{R}$, $A_i$ is  tropically nonsingular for any $0\leq i \leq n+1.$ Then by Lemma \ref{L3.4} and Lemma \ref{C2.13}, we have $H_{1}\cap \dots\cap H_{i-1}\cap H_{i+1}\cap\dots \cap H_{n+2}=\varnothing$.\par

 Combining \eqref{eq111}, \eqref{eq33} and the Definition \ref{D1.3}, we obtain that the tropical hyperplanes \(H_1, H_2, \dots, H_{n+2}\) in \(\mathbb{TP}^n\) are in general position. Thus by Theorem \ref{T4.1.2} and the first main theorem, we have
\[
\begin{split}
T_{f}(r)&\leq\sum_{j = 1}^{n+2} N(r, \frac{1_{\mathbb{T}}}{P_j\circ f} \oslash ) -\sum_{j = 0}^{n}N(r, 1_{\mathbb{T}} \oslash f_j)+ O(1)\\
&= N(r, \frac{1_{\mathbb{T}}}{P_{n+2}\circ f} \oslash )+ O(1)=N(r, \frac{1_{\mathbb{T}}}{P\circ f} \oslash)+ O(1)\leq T_{f}(r) + O(1).
\end{split}
\]
Thus, the proof is completed.
\end{proof}

To give an alternative proof of Theorem \ref{T4.1.10}, we obtain the following proposition.\par
\begin{proposition}
\label{C3.9}
Assume that the tropical holomorphic curve \(f = [f_0:\dots:f_n]:\mathbb{R}\to\mathbb{TP}^n\) is tropically linearly nondegenerate. If
$$\limsup_{r \to \infty} \frac{\log T_f(r)}{r}=0,$$
then
\[
\begin{split}
 N(r, 1_{\mathbb{T}} \oslash C(f))\leq \sum_{j = 0}^{n}N(r, 1_{\mathbb{T}} \oslash f_j)+o(T_f(r)),
\end{split}
\]
where $r$ tends to infinity outside of a set of zero upper density measure.
\end{proposition}

\begin{proof}
By the tropical Jensen formula, we have
\begin{equation}\label{eq3.2.6}
\begin{split}
N(r, 1_{\mathbb{T}} \oslash C(f))&-\sum_{j = 0}^{n}N(r, 1_{\mathbb{T}} \oslash f_j)\\
&=\frac{1}{2}\sum_{\sigma = \pm 1} C(f_0,\dots,f_n)(\sigma r)-\frac{1}{2}\sum_{j = 0}^{n}\sum_{\sigma = \pm 1}f_j(\sigma r)+O(1)\\
&=\frac{1}{2}\sum_{\sigma = \pm 1}\left[\frac{C(f_0,\dots,f_n)}{f_0\otimes\dots\otimes f_n}\oslash\right](\sigma r)+O(1).
\end{split}
\end{equation}
Below, we estimate the upper bound of \eqref{eq3.2.6}.
\[
\begin{split}
H:=\frac{C(f_0,\dots,f_n)}{f_0\otimes\dots\otimes f_n}\oslash=&\frac{\bigoplus(\overline{f_0}^{[\pi(0)]}\otimes\dots\otimes\overline{f_n}^{[\pi(n)]})}{f_0\otimes\dots\otimes f_n}\oslash\\
=&\left[\bigoplus\frac{\overline{f_0}^{[\pi(0)]}\otimes\dots\otimes\overline{f_n}^{[\pi(n)]}}{\overline{f_0}^{[\pi(0)]}\otimes\dots\otimes\overline{f_0}^{[\pi(n)]}}\right]\oslash\frac{f_0\otimes\dots\otimes f_n}{f_0\otimes\dots\otimes f_0}\\
&+\frac{\overline{f_0}^{[0]}\otimes\dots\otimes\overline{f_0}^{[n]}}{f_0\otimes\dots\otimes f_0}\\
:=&K+\frac{\overline{f_0}^{[0]}\otimes\dots\otimes\overline{f_0}^{[n]}}{f_0\otimes\dots\otimes f_0}.
\end{split}
\]
This is because $K$ consists purely of tropical sums and products of the form $(\overline{f_j}^{[m]}\oslash\overline{f_0}^{[m]}) \oslash ({f_j}\oslash {f_0})$ where $ m \in \{0, 1, \dots, n\}$ and $j \in \{0, \dots, n\}$. Therefore,
\begin{eqnarray}\label{eq3.2.7}
m(r,K) \leq \sum_{j=0}^n \sum_{m=0}^n m\left(r, \overline{f_{j}\oslash f_{0}}^{[m]}\oslash(f_j\oslash f_0)\right).\end{eqnarray}
We will estimate $m(r,K)$ below. Since $f_j\oslash f_0$ are tropical meromorphic functions for all $j\in\{1, 2, \dots, n\}$, by  Lemma \ref{L2.6}, we have\\
$$T_{f_j\oslash f_0}(r)\leq T_f(r)+ O(1).$$
This implies that
$$\limsup_{r\rightarrow\infty}\frac{\log T_{f_j\oslash f_0}(r)}{r}\leq\limsup_{r\rightarrow\infty}\frac{\log T_f(r)}{r}=0.$$
Hence, by Theorem \ref{T2.7}, we have
\begin{eqnarray}\label{eq3.2.8}
\sum_{j=0}^n \sum_{m=0}^n m\left(r, \overline{f_{j}\oslash f_{0}}^{[m]}\oslash(f_j\oslash f_0)\right)=o(T_f(r))
\end{eqnarray}
for $r\notin E$ with $\overline{dens}E = 0$.\par

As the same discussion as in the proof of Theorem \ref{T4.1.7}  we also have \eqref{eqq3.5}, \eqref{eqq3.6}, and then combining with \eqref{eq3.2.7} , \eqref{eq3.2.8}, we have
\[
\begin{split}
\frac{1}{2}\sum_{\sigma = \pm 1}H(\sigma r)=&\frac{1}{2}\sum_{\sigma = \pm 1}K(\sigma r)+\frac{1}{2}\sum_{\sigma = \pm 1}\frac{\overline{f_0}^{[0]}\otimes\dots\otimes\overline{f_0}^{[n]}}{f_0\otimes\dots\otimes f_0}(\sigma r)\\
\leq&\sum_{\sigma=\pm1}(K)^+(\sigma r)-\sum_{\sigma=\pm1}(-K)^+(\sigma r)+o(T_f(r))\\
=&m(r,K)-m(r,1_{\mathbb{T}} \oslash K)+o(T_f(r))\\
\leq& m(r,K)+o(T_f(r))\\
\leq&o(T_f(r)).
\end{split}
\]
Together this with \eqref{eq3.2.6}, the proof is completed.
\end{proof}

\noindent{\bf Alternative proof of Theorem \ref{T4.1.10}.} Combining Theorem \ref{T4.1.2} and Proposition \ref{C3.9}, one can immediately obtain  Theorem \ref{T4.1.10}.\qed\par

\section{Second main theorem for tropical meromorphic functions}
In this section, we will discuss the second main theorem for tropical meromorphic function. We observe that it is very interesting that for any $q$ values $a_1, \dots, a_q$ in $\mathbb{TP}^{1},$ being distinct is equivalent to being in general position in $\mathbb{TP}^{1}.$  \par

\begin{proposition}\label{P5.3}
$a_{1}, \dots, a_{q}$ are distinct values of $\mathbb{TP}^{1}$ if and only if they are in general position.
\end{proposition}

\begin{proof} Suppose that $a_{1}, \dots, a_{q}\in \mathbb{TP}^{1}$ are distinct. Denote $a_j:=[a_{j0}: a_{j1}].$ We will prove that $a_{1}, \dots, a_q$ are in general position. If not, then by Lemma \ref{C2.13} there exists \(\{i,j\}\subset\{1,\dots,q\}\) such that $a_i, a_j$ is tropically linearly dependent. So, by Lemma \ref{L3.4}, we get that
\[
\begin{bmatrix}
a_{i0} & a_{i1} \\
a_{j0} & a_{j1}
\end{bmatrix}
\] is tropically singular, that is, its determinant $(a_{i0}\otimes a_{j1}) \oplus (a_{j0}\otimes a_{i1})$ is attained at least twice. Then it is easy to find that $a_{i}=a_{j}$,  this is a contradiction. \par

Now assume that $a_{1}, \dots, a_q$ are in general position. We will prove that $a_{1}, \dots, a_{q}\in \mathbb{TP}^{1}$ are distinct. If not, then there exists \(\{i,j\}\subset\{1,\dots,q\}\) such that $a_{i}=a_{j}$, that is, $a_{j1}\oslash a_{j0}=a_{i1}\oslash a_{i0}$, then the maximum value of $a_{j1}\otimes a_{i0}\oplus a_{i1}\otimes a_{j0}$ is attained at least twice. Hence,
\[
\begin{bmatrix}
a_{i0} & a_{i1} \\
a_{j0} & a_{j1}
\end{bmatrix}
\] is tropically singular. By Lemma \ref{L3.4}, $a_i$ and $a_j$ are tropically linearly dependent. Thus by Lemma \ref{C2.13}, this is a contradiction.
\end{proof}

By Theorem \ref{T4.1.2} we have a corollary in one dimension as follows.\par

\begin{corollary}\label{T5.3} Let $f$ be a nonconstant tropical meromorphic function, and let $a_{j}=[a_{j0}: a_{j1}]$ ($j = 1,\dots, q$) be distinct values of $\mathbb{TP}^{1}$ defining tropical polynomials $P_{j}(x)=a_{j0}\otimes x_{0}\oplus a_{j1}\otimes x_{1},$ respectively.
Then
\[
\begin{split}
(q - 2)T_{f}(r)&\leq\sum_{j = 1}^{q}N(r,\frac{1_{\mathbb{T}}}{P_{j}\circ f}\oslash)-\sum_{j = 0}^{1}N(r, 1_{\mathbb{T}} \oslash f_j)+O(1).
\end{split}
\]
\end{corollary}

\begin{proof} First we need to prove that $f=[f_0: f_1]$ is  tropically linearly nondegenerate. Otherwise, there exist $b_0$ and $b_1$ such that
\begin{eqnarray*}\max\{b_{0}\otimes f_{0}, b_{1}\otimes f_{1}\}
\end{eqnarray*} is attained at least twice, that is, $b_{0}\otimes f_{0}=b_{1}\otimes f_{1}.$ This means $f=f_{1}\oslash f_{0}=b_{0}\oslash b_{1},$ which contradicts to the assumption that $f$ is nonconstant.\par

By Proposition \ref{P5.3},  $a_{j}$ ($j = 1,\dots, q$) are in general position. Hence, by Theorem \ref{T4.1.2} we obtain the conclusion of the corollary.
\end{proof}

For any $a=[a_1:a_0]\in \mathbb{R}\cup\{-\infty\} $, we can find that the maximum value of $P\circ f=(a_0\otimes f_0)\oplus(a_1\otimes f_1)$ is attained at least twice is equivalent to $f\oplus a$ is attained at least twice. Hence, by the tropical Jensen formula, we have
\begin{equation*}\label{eq333}
\begin{split}
N&\left(r, \frac{1_{\mathbb{T}}}{P \circ f} \oslash\right)= \frac{1}{2}\bigl(P(f)(r) + P(f)(-r)\bigr) + O(1) \\
&= \frac{1}{2}\bigl((a_0 \otimes f_0(r)) \oplus (a_1 \otimes f_1(r)) + (a_0 \otimes f_0(-r)) \oplus (a_1 \otimes f_1(-r))\bigr) + O(1) \\
&= \frac{1}{2}\bigl((a_0 \oslash a_1) \oplus (f_1(r) \oslash f_0(r)) + (a_0 \oslash a_1) \oplus (f_1(-r) \oslash f_0(-r))\bigr) \\
&\quad + \frac{1}{2}\bigl(f_0(r) + f_0(-r)\bigr) + O(1)\\
&= \frac{1}{2}\bigl(a \oplus f(r)\bigr)+\frac{1}{2}\bigl(a \oplus f(-r)\bigr)+ \frac{1}{2}\bigl(f_0(r) + f_0(-r)\bigr) + O(1)\\
&= N\left(r, \frac{1_{\mathbb{T}}}{f \oplus a} \oslash\right) - N(r, f \oplus a)+ N\left(r, \frac{1_{\mathbb{T}}}{f_0} \oslash\right)- N(r, f_0) + O(1)\\
&= N\left(r, \frac{1_{\mathbb{T}}}{f \oplus a} \oslash\right) + N(r, f) - N(r, f \oplus a) + O(1).
\end{split}
\end{equation*}

By Corollary \ref{T5.3} and the above equation, we can obtain the following second main theorem without any growth condition and exceptional set, which improves and generalizes Lain-Tohge's second main theorem \cite{5}.\par

\begin{theorem}\label{E5.4} Assume that $f$ is a nonconstant tropical meromorphic function, if $q (\geq1)$ distinct values $a_{1},\dots,a_{q}\in\mathbb{R}\cup\{-\infty\}$, then
\begin{equation*}
\begin{split}
(q - 2)T(r,f)+&N(r,f)+N(r,1_{\mathbb{T}}\oslash f)\\
&\leq\sum_{j = 1}^{q} \left(N(r, \frac{1_{\mathbb{T}}}{f \oplus a_j} \oslash) + N(r, f) - N(r, f \oplus a_j)\right)+O(1).
\end{split}
\end{equation*}
 In addition, if
\begin{equation*}
\max\{a_{1},\dots,a_{q}\}<\inf\{f(\alpha):\omega_{f}(\alpha)<0\},
\end{equation*}
then
\[
(q - 2)T(r,f)+N(r,f)+N(r,1_{\mathbb{T}}\oslash f)\leq\sum_{j = 1}^{q} N(r, \frac{1_{\mathbb{T}}}{f \oplus a_j} \oslash)+O(1).
\]
\end{theorem}

Lastly, we recall the classical Nevanlinna's truncated second main theorem for meromorphic functions in the complex plane
\[
(q - 1)T_f(r)\leq \sum_{j = 1}^{q}N^{(1)}(r,\frac{1}{f - a_j})+N^{(1)}(r,f)+o(T_f(r))
\]
for all $r$ possibly outside a set with finite linear measure, where $a_1,\dots,a_q$ are
distinct values in $\mathbb{C}$ and $N^{(1)}(r,\frac{1}{f - a_j})$ is the counting function of zeros
of $f - a_j$ with multiplicities truncated at level one. Here, we will take an example to show that  classical Nevanlinna's truncated second main theorem may fail in the tropical theory.\par

\begin{example}\label{E5.7} Let \(f_0=(-2x+12)\oplus (3x-18), f_1=(-3x)\oplus 0 \oplus (4x-8)\). Then
\[
f(x):=f_1(x)\oslash f_0(x)=
\begin{cases}
-x-12, & x < 0, \\
2x-12, & 0\leq x\leq 2, \\
6x-20, & 2< x\leq 6, \\
x+10, & x>6.
\end{cases}
\]
Take $a_1=-12$, $a_2=-2.$ Then as $r>6$  we have
\[
\begin{split}
m(r,f)&=\frac{f^{+}(r)+f^{+}(-r)}{2}=r-1,\\
N(r,f)&=\frac{5}{2}r-15, N^{(1)}(r,f)=\frac{1}{2}r-3,\\
N^{(1)}(r,\frac{1_{\mathbb{T}}}{f\oplus a_{1}}\oslash)&=r-1, N^{(1)}(r,\frac{1_{\mathbb{T}}}{f\oplus a_{2}}\oslash)=r-4,\\
T(r,f)&=m(r,f)+N(r,f)=\frac{7}{2}r-16.
\end{split}\]Thus as $r>6$ we get the estimation
\[
T(r,f)>
 \sum_{j = 1}^{2}N^{(1)}(r,\frac{1_{\mathbb{T}}}{f \oplus a_j}\oslash)+N^{(1)}(r,f)+O(1).
\]
This implies that there is no truncated Second main Theorem in tropical setting.
\end{example}

\section{Further remarks on degeneracy}
In this brief section we note a connection between the general position condition and the degeneracy parameters introduced by Korhonen-Tohge \cite{6}. Recall that Theorem \ref{T1.2} involves the degeneracy $\lambda = ddg(\{P_{n+2}\circ f,\dots,P_q\circ f\})$, which measures how many of the remaining hyperplanes are non-complete. However, in the general position setting of Theorem \ref{T4.1.2}, the coefficients $a_{\mu(0)},\dots,a_{\mu(n)}$ are tropically linearly independent, which imposes strong constraints on the hyperplanes. A natural question is: under the general position hypothesis, what can we say about the degeneracy of the hyperplanes? \par

We define a special degeneracy $ddg^*(Q)$ under the tropically linearly independence.\par

\begin{definition}\label{DD4.9} Let $F = \{f_0,\dots,f_n\}$ be a set of tropical entire functions, tropically linearly independent, and let $Q$ be a collection of tropical linear combinations of F over $\mathbb{T}$. The \textit{degree of 1-degeneracy} of Q is defined to be
$$ddg^*(Q)=card(\{g\in Q:\ell(g)= 1\}),$$
that is, the number of elements with precisely one real coefficient in their tropical linear combinations.
\end{definition}

The following result shows that if the first $n+1$ hyperplanes are \emph{maximally $1$-degenerate} in the sense that each $P_i\circ f$ reduces to a single component (so that $\lambda^* = ddg^*(\{P_1\circ f,\dots,P_{n+1}\circ f\})=n+1$), then the remaining hyperplanes $H_{n+2},\dots,H_q$ must be complete, i.e., $\lambda=0$. Consequently, the inequality in Theorem \ref{T4.1.2} becomes an exact equality \eqref{Eq1111}, and Theorem \ref{Tt3.12} can be applied. Thus, the general position condition together with maximal $1$-degeneracy forces the rest of the hyperplanes to be complete. This observation is not needed for the main results but may be of independent interest.\par

\begin{theorem}\label{CC4.10} Assume that the tropical holomorphic curve \(f = [f_0:\dots:f_n]:\mathbb{R}\to\mathbb{TP}^n\) is tropically linearly nondegenerate. Let \(H_1,\dots,H_q\) be tropical hyperplanes in \(\mathbb{TP}^n\) in general position defined by tropical linear polynomials $P_j$ ($j = 1,\dots,q$), respectively. If $\lambda^{*} = ddg^*(\{P_1\circ f,\dots,P_{q}\circ f\}),$ then
\begin{equation}\label{Eq111}
\begin{split}
(q - n - 1 )T_{f}(r)&\leq\sum_{j = 1}^{q} N(r, \frac{1_{\mathbb{T}}}{P_j\circ f} \oslash )-\sum_{j = 0}^{n}N(r, 1_{\mathbb{T}} \oslash f_j)+ O(1) \\
&\leq (q - \lambda^{*})T_{f}(r)+ O(1).
\end{split}
\end{equation}
In particular, when $\lambda^{*}=ddg^*(\{P_{1}\circ f,\dots,P_{n+1}\circ f\})=n+1$ (and so $\lambda=ddg(\{P_{n+2}\circ f,\dots,P_{q}\circ f\})=0$), we have
\begin{equation}\label{Eq1111}
\begin{split}
(q - n - 1 )T_{f}(r)&=\sum_{j = n+2}^{q} N(r, \frac{1_{\mathbb{T}}}{P_j\circ f} \oslash )+O(1).
\end{split}
\end{equation}
\end{theorem}

\begin{proof}
By Theorem \ref{T4.1.2}, the inequality $$(q - n - 1 )T_{f}(r)\leq\sum_{j = 1}^{q} N(r, \frac{1_{\mathbb{T}}}{P_j\circ f} \oslash )-\sum_{j = 0}^{n}N(r, 1_{\mathbb{T}} \oslash f_j)+ O(1)$$ is obtained. Hence, we only need to prove the case whenever $\lambda^{*} \neq 0.$ According to the definition of $ddg^{*},$  there are $\lambda^{*}$ terms in
\[
\left\{N(r, \frac{1_{\mathbb{T}}}{P_1\circ f} \oslash ), \dots,N(r, \frac{1_{\mathbb{T}}}{P_q\circ f} \oslash )\right\}
\]  that can cancel out some terms in $\{N(r, 1_{\mathbb{T}} \oslash f_0), \dots, N(r, 1_{\mathbb{T}} \oslash f_n)\}$. Without loss of generality, we assume that $P_i\circ f,$ $(1\leq i\leq\lambda^{*}),$ satisfy $\ell(P_i\circ f)= 1,$ that is,
$$P_{i}\circ f:=a_{i,i-1}\otimes f_{i-1} \quad \quad  (1\leq i\leq\lambda^{*})$$
where $a_i=(0_{\mathbb{T}},\dots,a_{i,i-1},\dots,0_{\mathbb{T}}), a_{i,i-1}\in \mathbb{R}$. By the tropical Jensen formula, we have
\begin{equation*}
\begin{split}
N(r, \frac{1_{\mathbb{T}}}{P_j\circ f} \oslash)=&\frac{1}{2}(P_i(f)(r)+P_i(f)( - r))+O(1)\\
=&\frac{1}{2}(a_{i,i-1}\otimes f_{i-1}(r)+a_{i,i-1}\otimes f_{i-1}( - r))+O(1)\\
=&\frac{1}{2}(f_{i-1}(r)+f_{i-1}(-r))+O(1)\\
=&N(r,1_{\mathbb{T}} \oslash f_{i-1})+O(1)
\end{split}
\end{equation*}for $1\leq i\leq\lambda^{*}.$ Then we have
\begin{equation*}
\begin{split}
\sum_{j = 1}^{q} N(r, \frac{1_{\mathbb{T}}}{P_j\circ f} \oslash)&-\sum_{j = 0}^{n}N(r, 1_{\mathbb{T}} \oslash f_j)\\
&=\sum_{j = \lambda^{*}+1}^{q} N(r, \frac{1_{\mathbb{T}}}{P_j\circ f} \oslash)-\sum_{j = \lambda^{*}}^{n}N(r, 1_{\mathbb{T}} \oslash f_j)+O(1)\\
&\leq(q - \lambda^{*})T_{f}(r)+O(1).
\end{split}
\end{equation*} Therefore, we get the inequality \eqref{Eq111}. In particular, whenever $\lambda^{*}=ddg^*(\{P_{1}\circ f,\dots,P_{n+1}\circ f\})=n+1,$ it is easy to see that
$$(q - n - 1 )T_{f}(r)=\sum_{j = n+2}^{q} N(r, \frac{1_{\mathbb{T}}}{P_j\circ f} \oslash )+O(1).$$\par

When $\lambda^{*}=ddg^*(\{P_{1}\circ f, \dots, P_{n+1}\circ f\})=n+1.$ Note that \(H_1, \dots, H_q\) in \(\mathbb{TP}^n\) are in general position. Without loss of generality, we set $$P_{i}\circ f:=a_{i-1}\otimes f_{i-1} \quad (1\leq i\leq\ n+1).$$ We claim $\lambda=ddg(\{P_{n+2}\circ f,\dots,P_{q}\circ f\})=0.$  Otherwise, we suppose that exist some $P_{j}\circ f$ $(j\in\{n+2, \dots,  q)$ is non-complete, that is  $$P_{j}\circ f=a_{j0}\otimes f_0\oplus\dots\oplus a_{jn}\otimes f_n$$ where at leat one of $a_{j0},\dots,a_{jn}$ is equal to $0_{\mathbb{T}}.$ We may assume that $a_{jn}=0_{\mathbb{T}}.$ Then for tropical hyperplanes $H_{1}, \dots, H_{n}$ and the $H_{j}$, the determinant of the matrix formed by their coefficient vectors is
\[|A_{n+1}|:=
\begin{bmatrix}
a_0 & 0_{\mathbb{T}} & \dots & 0_{\mathbb{T}} & 0_{\mathbb{T}} \\
0_{\mathbb{T}} & a_1 & \dots & 0_{\mathbb{T}} & 0_{\mathbb{T}}\\
\vdots & \vdots & \ddots & \vdots & \vdots\\
0_{\mathbb{T}} & 0_{\mathbb{T}} & \dots & a_{n-1} & 0_{\mathbb{T}}\\
a_{j0} & a_{j1} & \dots & a_{j,n-1}& 0_{\mathbb{T}}
\end{bmatrix}=0_{\mathbb{T}}.\] Hence, $A_{n+1}$ is  tropically singular, by Lemma \ref{L3.4} and Lemma \ref{C2.13}, we deduce that $H_{1}\cap \dots\cap H_{n}\cap H_{j}\neq\varnothing.$  This contradicts with \(H_1,\dots,H_q\) in \(\mathbb{TP}^n\) be in general position. Therefore, we obtain $\lambda=ddg(\{P_{n+2}\circ f,\dots,P_{q}\circ f\})=0.$
\end{proof}

\begin{remark} From the proof of Theorem \ref{CC4.10}, we know that if $\lambda^{*}=ddg^*(\{P_{1}\circ f,\dots,P_{n+1}\circ f\})=n+1,$ then $\lambda=ddg(\{P_{n+2}\circ f,\dots,P_{q}\circ f\})=0.$ This means that all the hyperplanes $H_{n+2}, \dots, H_q$ are complete, i. e. all coefficients of the hyperplanes are real numbers. Therefore the conclusion \eqref{Eq1111} can be obtained immediately by Theorem \ref{Tt3.12}.
\end{remark}

\section{Conclusion}
This paper refines and completes the tropical Cartan theory for hyperplanes in general position initiated by Korhonen-Tohge \cite{6} and Cao-Zheng \cite{7}. We have proved two second main theorems.  The first (Theorem \ref{T4.1.10}) achieves the optimal coefficient $q-n-1$ under the subnormal growth condition, using a simpler Casorati term. The second (Theorem \ref{T4.1.2}) is the first result in tropical Nevanlinna theory that requires no growth condition and no exceptional set; it relies on a tropical Cramer theorem and completely avoids the logarithmic derivative lemma. Together, these results demonstrate that the optimal geometric content of general position is fully captured in the tropical setting, and that analytic growth restrictions are not intrinsic to the subject.\par

\end{document}